\documentclass[12pt,leqno,fleqn]{amsart}  
\usepackage[color=green!15,bordercolor=red]{todonotes}
\usepackage{amsmath,amstext,amsthm,amssymb,amsxtra}
\usepackage[normalem]{ulem}
\usepackage{txfonts} 
\usepackage[T1]{fontenc}
\usepackage{lmodern}

 \usepackage{euler}   

\usepackage{mathtools}
\mathtoolsset{showonlyrefs,showmanualtags}

\usepackage{hyperref} 
\hypersetup{
    colorlinks=true,       
    linkcolor=blue,          
    citecolor=magenta,        
    filecolor=magenta,      
    urlcolor=cyan           
}

\usepackage[]{amsrefs}

\allowdisplaybreaks
\swapnumbers

\theoremstyle{plain} 
\newtheorem{lemma}[equation]{Lemma} 
\newtheorem{proposition}[equation]{Proposition} 
\newtheorem{theorem}[equation]{Theorem}

\theoremstyle{definition}

\theoremstyle{remark}
\newtheorem{remark}[equation]{Remark}

\newtheorem*{ack}{Acknowledgment}

\numberwithin{equation}{section}

\title[Non-homogeneous Local $T1$]{Non-Homogeneous Local $T1$ Theorem: Dual Exponents} 

\author[MT Lacey]{Michael T. Lacey}   

\address[M.T.L]{ School of Mathematics, Georgia Institute of Technology, Atlanta GA 30332, USA}
\email {lacey@math.gatech.edu}
\thanks{Research supported in part by grant NSF-DMS 0968499, 
and  a grant from the Simons Foundation (\#229596 to Michael Lacey). 
 ML and AVV benefited from the research program Operator Related Function Theory and Time-Frequency Analysis at the Centre for Advanced Study at the Norwegian Academy of Science and Letters in Oslo during 2012--2013.
  AVV  was
supported by 
the Finnish Academy of Science and Letters, Vilho, Yrj\"o and
Kalle V\"ais\"al\"a Foundation.
}

\author[AV V\"ah\"akangas]{Antti V. V\"ah\"akangas}
\address[A.V.V.]{Department of Mathematics and Statistics,
P.O. Box 68, FI-00014 University of Helsinki, Finland} \email{antti.vahakangas@helsinki.fi} \thanks{}

\begin{document}

\begin{abstract}
We provide an alternative proof of a (local) $T1$ theorem 
for dual exponents in the non-homogeneous
setting of upper doubling measures. This previously known theorem provides
necessary and sufficient conditions for the $L^p$-boundedness
of Calder\'on--Zygmund operators in the described setting, and the  novelty
 lies in the method of proof.
\end{abstract}	
	
\maketitle

\setcounter{tocdepth}{1}
\tableofcontents

\section{Introduction}

\subsection{Background and motivation}
The subject of local $Tb$ theorems in the
classical setting of $\mathbf{R}^n$ 
with Lebesgue measure is rather well understood by now.
We refer, in particular, to \cite{hytonen_nazarov} and
to  \cites{MR1934198,1011.1747,MR2474120,0705.0840,lv-perfect,1209.4161}.
These theorems extend the 
David--Journ\'e $ T1$ Theorem \cite{MR763911}, and the $ Tb$ theorem of Christ \cite{MR1096400} by  giving 
  flexible conditions under which an operator $ T$ with a Calder\'on--Zygmund kernel 
 extends to a bounded linear operator on $ L ^2 $. 
 By `local' we understand that the $Tb$ conditions involve a family of test functions $b_Q$, one for each cube $Q$, which 
should satisfy a non-degeneracy condition on its `own' $Q$.
Furthermore, both $b_Q$ and $Tb_Q$ are subject to
normalized  integrability conditions on $Q$ (with suitable exponents). 
Symmetric assumptions are imposed on $T^*$.

In the non-homogeneous setting less is know.
In the relevant literature
\cite{1011.0642, 1201.0648, MR1909219}
one usually encounters stronger $L^\infty(\mathbf{R}^n)$ (sometimes $\mathrm{BMO}$) conditions on 
$Tb_Q$'s, as well as on test functions $b_Q$.
In the search after relaxation of these conditions
one faces complications  that
arise from the feature that the underlying measure
$\mu$
need not be doubling.

We provide an alternative proof of a  local $T1$ theorem---which is, in fact, a $T1$ theorem in its local formulation---in the
non-homogeneous setting of upper doubling measures,  \cite{0909.3231,0911.4387}.
The 
local testing functions are indicators of cubes: $b_Q=\mathbf{1}_Q$,
and integrability conditions 
on $\mathbf{1}_QT\mathbf{1}_Q$ and $\mathbf{1}_QT^*\mathbf{1}_Q$
are those of dual exponents 
$1<p_1<\infty$ and $p_2=p_1/(p_1-1)$.
This result is already known and available in the literature, see Remark \ref{r.known}, and the
motivation stems from the fact that our novel proof possibly lends
itself to other situations.
In particular, 
a  non-homogeneous local $Tb$ theorem, say, for dual exponents, 
has not yet been established, and
it seems plausible that the new techniques in the present paper can be used to attack this open and difficult problem.

More precisely, our proof relies upon a so called corona decomposition,
 adapted to the maximal averages of given
two functions $f_1$ and $f_2$.  The advantage of this approach is 
that one has powerful quasi orthogonality inequalities, 
useful throughout the proof.  A direct argument can be used 
to control  a difficult  `inside' term, thereby we
avoid the typical use of paraproducts and Carleson 
measures. This argument can be viewed as an 
extension of its
`homogeneous' counterparts that are developed in  \cite{lv-perfect,1209.4161}.

\subsection{A local $T1$ theorem}
Let
 $ \mu $ be a compactly supported Borel measure on $ \mathbf{R} ^{n}$.
We assume  
the upper doubling conditions of  Hyt\"onen \cite{0909.3231}: there is   
a dominating function $ \lambda \;:\; \mathbf{R} ^{n} \times \mathbf{R} _+ \to \mathbf{R} _+$, and a constant $C_\lambda>0$,
such that for all $x\in\mathbf{R}^n$ and $r>0$:
\begin{align}
\mu (B (x,r)) \le \lambda (x,r)\le C_\lambda \lambda(x,r/2)\,.
	\end{align}
Moreover, we assume that
$r\mapsto \lambda (x, r)$ is non-decreasing for all $ x\in \mathbf{R} ^{n}$.
 The number $ d = \log_2 C_ \lambda $ can be thought 
of as the dimension of $ \mu $.  

We assume that a linear operator $ T$
is bounded on $L^2(d\mu)$, and it is adapted to $ \lambda$ in the following sense.  There is 
a kernel $ K\;:\; \mathbf{R} ^{n} \times \mathbf{R} ^{n} \to \mathbf{R} $ such that for all compactly supported $f\in L^2(\mathbf{R}^n)$,
\begin{gather*}
Tf(x)=\int_{\mathbf{R}^n}  K (x,y) f (y) \, d\mu(y)\,,\qquad x\not\in\mathrm{supp}(f).
\end{gather*}
We assume that these kernel estimates hold for some $\eta\in (0,1)$: 
\begin{align}\label{e.size}
	\lvert  K (x,y)\rvert \le 
	\min\bigg\{\frac{1}{\lambda (x, \lvert  x-y\rvert )},
	\frac{1}{\lambda (y, \lvert  x-y\rvert )} \bigg\} \,, \qquad x \neq y\,,
\end{align}
\begin{equation}\label{e.cancel}
\lvert K(x,y)-K(x',y)\rvert \le \frac{\lvert x-x'\rvert^{\eta}}{\lvert x-y\rvert^{\eta} \lambda(x,\lvert x-y\rvert)},\qquad
\lvert x-y\rvert \ge 2\lvert x-x'\rvert\,,
\end{equation}
and
\[
\lvert K(x,y)-K(x,y')\rvert \le \frac{\lvert y-y'\rvert^{\eta}}{\lvert x-y\rvert^{\eta} \lambda(y,\lvert x-y\rvert)},\qquad
\lvert x-y\rvert \ge 2\lvert y-y'\rvert\,.
\]
The operator $T$ is said to be a {\em Calder\'on--Zygmund operator}.
We are interested in {\em quantitative} estimates for the operator norm of $T$ on $L^p(\mu)$ for  $1<p<\infty$, and the following hypothesis, together with
kernel assumptions, provides
the essential quantitative information.

\begin{itemize}
\item
{\em Local Testing Condition Hypothesis}.
For given two exponents $p_1,p_2\in (1,\infty)$, there is a constant $\mathbf{T}_{\textup{loc}}$ as follows.
For all  cubes $ Q$ in $\mathbf{R}^n$,
	\begin{equation}\label{l1_testing}
		\int _{Q} \lvert  T \mathbf 1_{Q}\rvert^{p_1} \, d\mu (x) \le \mathbf T_{\mathbf{loc}}^{p_1} \mu( Q)\,, 
		\qquad 
		\int _{Q} \lvert  T ^{\ast}  \mathbf 1_{Q}\rvert^{p_2} \, d\mu (x) \le \mathbf T_{\textup{loc}}^{p_2} 
		\mu(Q)\,.
\end{equation}
\end{itemize}

We provide a novel proof of the following  previously known theorem.
\begin{theorem}\label{t.main} 
Let $T$ be a  Calder\'on--Zygmund operator. 
Fix $1<p_1,p_2<\infty$, $1/p_1+1/p_2\le 1$. Assume  the
following two conditions (1)--(2):
\begin{itemize}
\item[(1)]  $T$ is (a priori) bounded on $L^{p_1}(d\mu)$;
\item[(2)] $T$ satisfies a Local Testing Condition Hypothesis with exponents $p_1$ and $p_2$.
\end{itemize}
Under these assumptions, we have a quantitative norm estimate 
\[\mathbf{T}:= \lVert T\rVert_{L^{p_1}(d\mu)\to L^{p_1}(d\mu)} \lesssim 1+\mathbf T_{\textup{loc}}\,,\]
where the implied constant depends on $n,p_1,p_2,\eta,\mu$.
\end{theorem}

In the sequel, unless otherwise specified, we  assume that $p_1$ and $p_2$ are in duality:
$p_2=\frac{p_1}{p_1-1}$.

\begin{remark}\label{r.known}
Theorem \ref{t.main} is known and available in the literature.
Indeed, under the assumptions of this theorem, it is straightforward to
verify that $T$ satisfies a `weak boundedness property'
and  `testing conditions', namely for all cubes $Q$ in $\mathbf{R}^n$, and an appropriate $\sigma\ge 1$,
\begin{equation}\label{e.T1_testing}
\bigg\lvert \int_{Q}T\mathbf{1}_Q\,d\mu\bigg\rvert \le \mathbf{T}_{\textup{loc}} \mu(Q)\,,\qquad 
T\mathbf{1} \in \mathrm{BMO}_\sigma^{p_1}(\mu),\qquad T^*\mathbf{1}\in \mathrm{BMO}_\sigma^{p_2}(\mu)\,;
\end{equation}
we refer to
Remark \ref{r.local_T1} for further details.
It remains to apply a non-homogeneous $T1$ theorem, see \cite{NTV1} or \cite[$Tb$ theorem 2]{0809.3097}
for $\lambda(x,r)=r^d$ dominating the measure,
and \cite[Theorem 2.1]{MR2990130} for the general case.
Moreover, by using the last theorem, it is even possible to relax
the integrability conditions in \eqref{l1_testing} to exponents  $p_1=1=p_2$.
Let us also remark that the case of $p_1=2=p_2$ has been addressed
in \cite{MR2956255} with  a 
function $\lambda(x,r)=\max\{\delta(x)^d,r^d\}$ dominating the measure,
where $\delta(x)=\mathrm{dist}(x,\mathbf{R}^n\setminus H)$ for an
an open set $H$ in $\mathbf{R}^n$.
\end{remark}

\begin{remark}
The $p$-independence property of Calder\'on--Zygmund operators,
i.e., if their $L^2$ boundedness is equivalent to their $L^p$ boundedness, 
has been addressed, for instance, in \cite{MR2957235,HuMengYang}.
It is an interesting question, if our proof can be
adapted to obtain a quantitative $p$-independence result for
Calder\'on--Zygmund operators, under an appropriate set of local testing hypotheses.
\end{remark}

\subsection{Structure of the paper}
We  use the  non-homogeneous techniques of \cite{NTV1},  in particular, 
 good and bad cubes are applied in a partially novel manner.
Martingale techniques,
 including $L^p$ estimates for martingale transforms and Stein's
ineguality, 
are fundamental.
These techniques are also applied in a related paper \cite{1201.0648},
from which 
we borrow also some other ideas, e.g.,
 treatments
of `separated' and  `nearby' terms.
Our main technical contribution  is treatment of the most difficult `inside' term
by  a strong definition of goodness and a corona decomposition, avoiding
(a) explicit   construction of paraproduct operators; and  (b)
 Carleson embedding theorems.

The heart of the matter
is  estimation of a
form $\lvert \langle Tf_1,f_2\rangle\rvert $,
where $f_j$'s are perturbed functions, supported
on large dyadic cubes $Q_{j,0}\in \mathcal{D}_j$.
Here $\mathcal{D}_j$ is a random dyadic system.
The perturbation 
 is  simply a projection to good cubes, and results 
in that the usual martingale differences 
$\Delta_Q f_j$ vanish if $Q\subset Q_{j,0}$ is a bad.
After a probabilistic absorption argument, the focus will be on a triangular form
\[
\Big\lvert
\sum_{\substack{P,Q\textup{ good}}}
\mathbf{1}_{\ell Q\le \ell P}\cdot
\langle T\Delta_P f_1, \Delta_Q f_2\rangle\Big\rvert\,,
\]
where always $P\subset Q_{1,0}$ and $Q\subset Q_{2,0}$.
This form is further split into `inside', `separated', and
`nearby' terms.
The analysis of the inside term,
in which $Q$ is deeply inside $P$, is taken up in
sections
\ref{s.inside_core} and \ref{s.remaining}---the argument is transparent,
and our strong definition of {\em goodness of cubes} has a key role.
The construction of paraproducts is avoided, and even Carleson embedding theorems are
not needed; in this we follow \cites{1108.2319,1201.4319}.
We apply a corona decomposition, and the
associated stopping tree 
is  constructed in Section \ref{s.corona}, where
we also record the basic  `quasi-orthogonality' properties.
The separated
 term,
 in which $Q$ is always far away from $P$,
  is analysed in Section \ref{s.separated}, and the (usual) goodness is crucial.
Throughout sections \ref{s.nearby1}--\ref{s.nearby3}, we treat the nearby term,
where cubes are close to each other both in position
and size.  The usual surgery is performed.

\begin{ack}
The authors would like to thank Tuomas Hyt\"onen and Henri Martikainen for indicating
the connection of Theorem \ref{t.main} to the $T1$ theorems that are available  in the literature.
\end{ack}

\section{Preliminaries}

\subsection{Notation}
 The implied constants are allowed to depend
 upon parameters $r,n,p_1,p_2,\eta,\mu$.
The distances are measured in supremum norm, 
$\lvert x\rvert = \rVert x\rVert_\infty$ for $x\in\mathbf{R}^n$.
We denote $L^p=L^p(d\mu)$ if $1\le p\le\infty$.
For a cube $ Q$ and $f\in L^1_{\textup{loc}}$,  write
$ \langle f  \rangle_Q :=\mu(Q)^{-1} \int _{Q} f \; d\mu  $
with the convention $\langle f\rangle_Q=0$ if $\mu(Q)=0$.
The side length of a cube $Q$ is written as $\ell Q$, and
the midpoint as $x_Q$.
The `long distance' between cubes
 $Q$ and $P$ is
 $D(Q,P)=\ell Q + \mathrm{dist}(Q,P)+\ell P$.
 

A `dyadic cube'  is any cube in either random grid $\mathcal{D}_j$ with $j\in \{1,2\}$, 
Section \ref{s.random_grid}.
By $\mathcal{D}_{j,k}$ we denote those dyadic  cubes $Q\in\mathcal{D}_j$
for which $\ell Q=2^k$, $k\in \mathbf{Z}$.
The  dyadic children of $ Q\in\mathcal{D}_j$ are
$\{Q_1,\ldots,Q_{2^n}\}=\textup{ch}(Q)$, its
dyadic parent is $\pi_j Q=\pi_j^1 Q$, and
 $\pi_j^t Q = \pi_j (\pi_j^{t-1}Q)$ for $t\in \{2,3,\ldots\}$.
For $ \mathcal S_j\subset \mathcal D_j$ the family $ \textup{ch} _{\mathcal S_j} (S)
=\textup{ch}^1 _{\mathcal S_j} (S)$ consists of the $ \mathcal S_j$-children of $ S\in \mathcal S_j$: 
the maximal cubes in $ \mathcal S_j$ that are strictly contained in $ S$.
We also denote $\textup{ch}_{\mathcal{S}_j}^0(S)=\{S\}$ and,
for $t>1$, write
$S'\in\textup{ch}^t_{\mathcal{S}_j}(S)$ if $S'\in\textup{ch}_{\mathcal{S}_j}(S'')$ for some $S''\in \textup{ch}^{t-1}_{\mathcal{S}_j}(S)$.
For {\em any} cube $ Q$ which is contained in a cube in $\mathcal{S}_j$,
we take $ \pi _{\mathcal S_j} Q=\pi _{\mathcal S_j}^0 Q$ to be the $ \mathcal S_j$-parent of $ Q$: the minimal 
$\mathcal{S}_j$-cube 
containing $ Q$
(if $Q$ is not contained in a cube in $\mathcal{S}_j$, we set $\pi_{\mathcal{S}_j}Q=\mathbf{R}^n$).
For $t\ge 1$ and any cube $Q$, contained in at least $t+1$ cubes in $\mathcal{S}_j$,  we 
let
$\pi _{\mathcal S_j}^t Q$ to be $\pi_{\mathcal{S}_j}^{t-1} S'$, where 
$\pi_{\mathcal{S}_j}Q\in\textup{ch}_{\mathcal{S}_j}(S')$.

\subsection{Random grids}\label{s.random_grid}
We use the foundational tool of random grids, initiated by Nazarov--Treil--Volberg \cite{MR1909219},
which has in turn been used repeatedly. 
We refer, e.g., to \cite{MR1756958,1201.4319,0911.4387,MR2912709}. 
Throughout the paper, we shall use two random dyadic grids  (systems) $\mathcal {D}_j$, $j\in \{1,2\}$. A third random grid $\mathcal{D}_3$ appears
at the very end.
These are constructed as follows; we refer to \cite{1108.5119} for further details.

The random grids $\mathcal{D}_j$ are parametrized by  sequences $\omega_j\in (\{0,1\}^n)^{\mathbf{Z}}$,
$j\in \{1,2,3\}$,
where we tacitly assume three independent copies
of $(\{0,1\}^n)^{\mathbf{Z}}$.
More precisely, for a cube  $\widehat{Q}\in\widehat{\mathcal{D}}$ in the standard dyadic grid, the
\emph{position} of an $\omega_j$-translated cube is
\begin{equation*}
	Q=\widehat{Q} \dot+ \omega_j := \widehat{Q}+\sum_{k\,:\,2^{-k}< \ell \widehat{Q} } 2^{-k}\omega_{j,k}\,,
\end{equation*}
which is a function of $ \omega_j\in (\{0,1\}^n)^{\mathbf{Z}} $.
A dyadic grid (system)
\[
\mathcal{D}_j=\mathcal{D}(\omega_j)=\{\widehat{Q} \dot+ \omega_j\,:\,\widehat{Q}\in\widehat{\mathcal{D}}\}
\]
is the family  of these $\omega_j$-translated cubes. 
The natural uniform probability measure  $\mathbf{P}_{\omega_j}$  is placed upon the respective copy of
$(\{0,1\}^n)^{\mathbf{Z}}$. 
Each component $\omega_{j,k}$, $k\in\mathbf{Z}$, has an equal probability $2^{-n}$ of taking any of the $2^n$ values, and all components are independent of each other. The expectation
 with respect to $\mathbf{P}_{\omega_j}$ is denoted by $\mathbf{E}_{\omega_j}$. 
We will usually simply write $P$ or $S$ for a cube in $\mathcal{D}_1$,
and $Q$ or $R$ for a cube in $\mathcal{D}_2$,
instead of the heavier notation $ \widehat{Q} \dot+ \omega_j  $ with 
$\widehat{Q}\in\widehat{\mathcal{D}}$.

Choose, once and for all, a constant $\gamma\in (0,1)$ such that 
\begin{equation} \label{e.epsilon}
d\gamma/(1-\gamma)\le \eta/4,\quad  \gamma\le \frac{\eta}{2(d+\eta)}\,,\qquad d=\log_2 C_\lambda\,.
\end{equation} 
Here $\eta$ is the constant appearing in the kernel condition \eqref{e.cancel}.
We  also
denote
\begin{equation}\label{e.theta_def}
\theta(j)=\Big\lceil\frac{\gamma j+r}{1-\gamma}\Big\rceil\quad \text{ for }j=0,1,2,\ldots.
\end{equation}
Throughout  $r\in \mathbf{N}$ should be thought of  as a large integer, 
whose exact value is assigned later.

\subsection{Goodness of cubes}
We  impose a strong definition of goodness: by doing so, we ensure
that good cubes $Q\in \mathcal{D}_1\cup \mathcal{D}_2$ from either system  are always far away from the boundaries of much larger
cubes in {\em either} one of these two systems.

A cube $ Q\in\mathcal{D}_j$ is 
$k$-\emph{bad} for $j,k\in \{1,2\}$ if there is a cube $P\in\mathcal{D}_k$ such that
$\ell P\ge 2^r \ell Q$ and 
$ \textup{dist} (Q, \partial P) \le (\ell Q) ^{\gamma } (\ell P) ^{1- \gamma}$.
Otherwise, $Q$ is $k$-\emph{good}.
The following properties are  known, \cite{1108.5119}.
\begin{enumerate}
\item 
For $\widehat{Q}\in\widehat{\mathcal{D}}$, 
position and
$k$-goodness of $Q=\widehat{Q}\dot+\omega_j$ are independent random variables.
\item The probability
$ \pi _{j,k,\textup{good}} := \mathbf P_{\omega_k}  (\textup{$\widehat{Q} \dot+\omega_j$ is $k$-good})$ is independent of  $ \widehat Q\in\widehat{\mathcal{D}}$.
\item  $  \pi _{j,k,\textup{bad}}:=1- \pi _{j,k,\textup{good}}  \lesssim 2 ^{- \gamma r}$, with  implied constant independent of $r$.
\end{enumerate}
 A cube $ Q\in\mathcal{D}_j$ with $j\in \{1,2\}$ is 
\emph{bad} if it is $k$-bad for some $k\in \{1,2\}$.
Otherwise, we say that $Q$ is \emph{good}. To state this
condition otherwise, if $Q\in\mathcal{D}_j$ is good, we have
inequality
\[
(\ell Q)^{\gamma}(\ell P)^{1-\gamma} < \mathrm{dist}(Q,\partial P)\,,
\]
if $P\in \mathcal{D}_1\cup \mathcal{D}_2$ and $2^r\ell Q\le \ell P$.
 Define bad and good projections by $I= P _{j,\textup{bad}} + P _{j,\textup{good}}$, where 
 \begin{equation*}
 	 P  _{j,\textup{bad}} \phi := \sum_{Q \in \mathcal D_j \;:\;  \textup{$ Q$ is bad}} \Delta _{Q} \phi\,,
	 \qquad \phi\in L^q\quad (1<q<\infty)\,. 
\end{equation*}
Here $ \Delta _Q \phi=\sum_{Q'\in\textup{ch}(Q)}
\{ \langle \phi\rangle_{Q'} - \langle \phi\rangle_Q\} \mathbf{1}_{Q'}$ is the martingale difference
with respect to $\mu$. 
The following proposition is a straightforward modification of
\cite[Proposition 2.4]{1209.4161}.

\begin{proposition}\label{p.bad} For every $j\in \{1,2\}$ and $ 1< q < \infty $  there is a constant $ c_q >0 $ so that 
	\begin{equation}\label{e.bad}
		\mathbf{E}_{\omega_1} \mathbf{E}_{\omega_2}\lVert  P _{j,\textup{bad}} \phi \rVert_{q}^q \lesssim 2 ^{- \gamma r/c_q} \lVert \phi \rVert_{q}^q\,,
\end{equation}
where $\phi\in L^q$ is any function, independent of both random grids $\mathcal{D}_k$ with $k\in \{1,2\}$. Moreover, the implied constant
is independent of $r$.
\end{proposition}



\begin{proof}
We apply
Marcinkiewicz interpolation theorem to the linear
	operator \[P  _{j,\textup{bad}}:L^q(d\mu)\to L^q(\mathbf{P}_{\omega_1}\otimes \mathbf{P}_{\omega_2}\otimes d\mu)\,.\]
	The projection to bad cubes is a martingale transform: by inequality \eqref{e.classical_martingale}, the following inequality with no decay holds, 
	\begin{equation*}
		\mathbf E_{\omega_1} \mathbf E_{\omega_2}\lVert  P  _{j,\textup{bad}} \phi \rVert_{p} ^{p} 
		\le \sup_{\omega_1,\omega_2}\, \lVert  P  _{j,\textup{bad}} \phi \rVert_{p} ^{p}
		\lesssim  \lVert \phi \rVert_{p} ^{p} \,,\quad 1<p<\infty\,.
\end{equation*}
	Thus,
	it suffices to verify the claimed decay for $ q=2$. To this end, we have 
	by orthogonality of martingale differences,
	\begin{align*}
		\mathbf E_{\omega_1} \mathbf E_{\omega_2} \lVert  P _{j,\textup{bad}} \phi \rVert_2 ^2 
		&=\mathbf E_{\omega_1} \mathbf E_{\omega_2} \sum_{ \widehat{Q}\in\widehat{\mathcal{D}} }  
		\mathbf{1}_{\widehat{Q}\dot+\omega_j\textup{ is bad}}
		\lVert \Delta_{\widehat Q\dot+\omega_j} \phi \rVert_{2} ^2 
		\\&\le  \sum_{k=1}^2	\mathbf E_{\omega_1} \mathbf E_{\omega_2}
			\sum_{ \widehat{Q}\in\widehat{\mathcal{D}} } \mathbf{1}_{\widehat{Q}\dot+\omega_j\textup{ is $k$-bad}} 
		\lVert \Delta_{\widehat Q\dot+\omega_j} \phi \rVert_{2} ^2 \\
				&\le \sum_{k=1}^2 \pi_{j,k,\textup{bad}}\mathbf E_{\omega_1} \mathbf E_{\omega_2} \sum_{ \widehat{Q}\in\widehat{\mathcal{D}} }  
		\lVert \Delta_{\widehat Q\dot+\omega_j} \phi \rVert_{2} ^2 
		\le (\pi _{j,1,\textup{bad}}	 + \pi _{j,2,\textup{bad}})	   \lVert \phi \rVert_{2} ^2\,.
\end{align*}
In the third step, we used Fubini's theorem, linearity of expectation, and
the fact that $\lVert\Delta_{\widehat Q\dot+\omega_j} \phi\rVert_2^2$ and $k$-badness of $\widehat Q\dot+\omega_j$ are independent random variables.
\end{proof}

\subsection{Square function inequalities}

The martingale transform inequality is this, see e.g. \cite{MR744226}.
For all functions $f\in L^p$, and  constants satisfying 
$\sup_{Q\in\mathcal{D}_j} \lvert\varepsilon_Q\rvert\le 1$,
\begin{equation}\label{e.classical_martingale}
	\biggl\lVert  \sum_{Q\in \mathcal D_j} \varepsilon _Q \Delta_Q f
	 \biggr\rVert_{p}  \lesssim
	 \lvert\lvert f\rvert\rvert_{p},\quad 1<p<\infty\,,\quad j\in \{1,2\}\,.
\end{equation}
A consequence of Khintchine's inequality and inequality
\eqref{e.classical_martingale} is the following.
\begin{equation}\label{kahane}
\bigg\lVert \bigg(\sum_{k\in\mathbf{Z}}|\Delta_{j,k} f|^2\bigg)^{1/2}\bigg\rVert_{p}
\lesssim \lVert f \rVert_{p}\,,
\end{equation}
where $f\in L^p$ with $1<p<\infty$, and
$\Delta_{j,k} f = \sum_{Q\in\mathcal{D}_{j,k}} \Delta_Q f$ for $k\in\mathbf{Z}$ and $j\in \{1,2\}$.

We will use the following {\em Stein's inequality},  see e.g. \cite{MR830227}. 
For $1<p<\infty$ and $j\in \{1,2\}$,
\begin{equation}\label{stein}
\bigg\lVert \bigg(  \sum_{k\in \mathbf{Z}} 
\lvert \mathbf{E}_{j,k} f_k \rvert^2\bigg)^{1/2}\bigg\rVert_{p}
\lesssim \bigg\lVert
\bigg( \sum_{k\in \mathbf{Z}} \lvert f_k\rvert^2\bigg)^{1/2}\bigg\rVert_{p}\,,
\end{equation}
where $(f_k)_{k\in\mathbf{Z}}$ is {\em any} sequence
in $L^p(d\mu)$, $\mathbf{E}_{j,k} f= \sum_{Q\in\mathcal{D}_{j,k}}
\mathbf{E}_Q f$, and $\mathbf{E}_Q f=\langle f\rangle_Q \mathbf{1}_Q$.
We don't rely on Fefferman--Stein inequalities
for the vector-valued maximal function.
Stein's inequality is their replacement in the present, non-homogeneous, setting.

\subsection{Off-diagonal estimates}
Here we collect useful off-diagonal estimates.

\begin{lemma}\label{l.off_diagonal}
Let $Q\subset P\subset R$ be cubes in $\mathbf{R}^n$
such that $\ell Q\le \mathrm{dist}(Q,R\setminus P)$.
Then, 
\begin{equation}\label{e.off_diagonal}
\lvert T\mathbf{1}_{R\setminus P}(x)-T\mathbf{1}_{R\setminus P}(x_Q)\rvert
\lesssim \bigg(\frac{\ell Q}{\mathrm{dist}(Q,R\setminus P)}\bigg)^{\eta}\,,\qquad x\in Q\,.
\end{equation}
\end{lemma}

\begin{proof}
The kernel condition
\eqref{e.cancel} applies,
\begin{align*}
LHS\eqref{e.off_diagonal}\le
\int_{R\setminus P} \lvert K(x,y)-K(x_Q,y)\rvert \,d\mu(y)
\le \int_{R\setminus P} \frac{\lvert x-x_Q\rvert^\eta}{\lvert x-y\rvert^\eta \lambda(x,\lvert x-y\rvert)}\, d\mu(y)\,.
\end{align*}
Let us denote 
$\delta:=\mathrm{dist}(Q,R\setminus P)$ and
$A_j = \{y\in\mathbf{R}^n\,:\, 2^j \delta \le \lvert x-y\rvert
< 2^{j+1}\delta\}$ for $j\ge 0$.
Observe that $R\setminus P\subset \cup_{j=0}^\infty A_j$.
Since $A_j\subset B(x,2\lvert x-y\rvert)$ for
each $y\in A_j$, we can bound the last integral by
$C_\lambda(\ell Q)^{\eta}\sum_{j=0}^\infty (2^j\delta)^{-\eta}
\lesssim (\ell Q/\delta)^{\eta}$
as required.
\end{proof}

\begin{remark}\label{r.local_T1}
Let us
verify that the 
a priori boundedness of $T$ on $L^{p_1}$, and
Local Testing Condition Hypothesis,  
together imply the assumptions of a $T1$ theorem; namely,
conditions \eqref{e.T1_testing} with $\sigma = 3$.
The first condition therein is, indeed, a trivial consequence of inequality \eqref{l1_testing}.
Hence, it suffices to verify that $b:= T\mathbf{1}$ satisfies
$b\in \mathrm{BMO}^{p_1}_\sigma(\mu)$,  i.e.,
\begin{equation}\label{e.bmo}
\lVert b\rVert_{\mathrm{BMO}^{p_1}_\sigma(\mu)}:=  \sup_Q \ \bigg\{ \frac{1}{\mu(\sigma Q)} \int_Q \lvert b(x) - \langle b\rangle_Q\rvert^{p_1}\,d\mu(x)\bigg\}^{1/p_1} \lesssim 1+\mathbf{T}_{\textup{loc}}\,,
\end{equation}
where the supremum is taken over all cubes $Q$ in $\mathbf{R}^n$.
Indeed, a completely analogous argument then shows that $T^*\mathbf{1}\in\mathrm{BMO}^{p_2}_\sigma(\mu)$.

In order to verify inequality \eqref{e.bmo}, let us fix a cube $Q$ in which the supremum above is (almost) attained.
Let us then fix a large cube $R$ in $\mathbf{R}^n$, containing both $3Q$ and the compact support of the measure $\mu$. In particular,
$T\mathbf{1} = T\mathbf{1}_R\in L^{p_1}$, and we can estimate
\begin{align*}
\lVert b\rVert_{\mathrm{BMO}^{p_1}_\sigma(\mu)}^{p_1} &\lesssim \frac{1}{\mu(3Q)}\int_Q \lvert T\mathbf{1}_{R}(x)  - T\mathbf{1}_{R\setminus 3Q}(x_Q)\rvert^{p_1}\,d\mu(x)\\
&\lesssim \frac{1}{\mu(3Q)}\int_Q \lvert T\mathbf{1}_{3Q}\rvert^{p_1}\,d\mu  + \frac{1}{\mu(3Q)}\int_Q \lvert T\mathbf{1}_{R\setminus 3Q}  - T_{R\setminus 3Q}(x_Q)\rvert^{p_1}\,d\mu\,.
\end{align*}
By inequality \eqref{l1_testing}, the first term in the last line is dominated by $\mathbf{T}_{\textup{loc}}^{p_1}$. And
by Lemma \ref{l.off_diagonal}, the last term is seen to be bounded by $\lesssim 1$.
\end{remark}

In the following two lemmata, we write
$D(Q,P)/\ell P \sim 2^u$ if
$2^{u} < {D(Q,P)}/\ell P \le 2^{u+1}$.

\begin{lemma}\label{l.s_exists_s}
Suppose 
$P\in\mathcal{D}_{1,k}$ and  $Q\in\mathcal{D}_{2,k-m}$
is a good cube
 such that $D(Q,P)/\ell P\sim 2^u$, where
 $k\in\mathbf{Z}$ and $u,m\in \mathrm{N}_0$.
Then, we have $Q\subset \pi^{u+\theta(u+m)}_1 P$.
\end{lemma}

\begin{proof}
Denote  $t=u+\theta(u+m)\ge r$. 
By goodness, either $Q\subset \pi^t P$ or $Q\subset \mathbf{R}^n\setminus \pi^t P$. 
In the former case, we are done.
In the latter case, we obtain a contradiction. Indeed, by goodness,
\begin{align*}
(\ell Q)^\gamma (\ell \pi^t P)^{1-\gamma} <\mathrm{dist}(Q,\partial \pi^t P)
=\mathrm{dist}(Q,\pi^t P)\le D(Q,P)\le 2^{u+1}\ell P\,.
\end{align*}
Substituting $\ell Q=2^{k-m}$ and $\ell P =2^k$ yields $u+\theta(u+m)=t < u+\theta(u+m)$ after elementary manipulations.
This is
a contradiction.
\end{proof}

\begin{lemma}\label{l.s_estimate}
Suppose that $P$ and $Q$
are
as in Lemma \ref{l.s_exists_s}.
 Assume also that
$\ell Q\le \mathrm{dist}(Q,P)$.
 Then, by denoting $S:=\pi^{u+\theta(u+m)}_1P$, 
 \begin{equation}\label{e.kineq}
 \lvert K(x,y)-K(x_Q,y)\rvert \lesssim \frac{2^{-\eta(u+m)/4}}{\mu(S)},\qquad
 (x,y)\in Q\times P\,.
 \end{equation}
\end{lemma}

\begin{proof}
By inequalities \eqref{e.cancel} and $\ell Q\le \mathrm{dist}(Q,P)$,
we obtain
$LHS\eqref{e.kineq} \le  \alpha \cdot \beta$
where $\mathbf{\alpha}=C_\lambda^\kappa (\ell Q)^\eta/\mathrm{dist}(Q,P)^{\eta}$ and $\mathbf{\beta}=1/\mu( B(x, 2^\kappa \lvert x-y\rvert))$ with $\kappa$
specified in the two case studies.

Case $\ell P <\mathrm{dist}(Q,P)$. Choose $\kappa = 2 +\theta(u+m)$.
Observe the inequality $2^{u+k}< 4\mathrm{dist}(Q,P)$. 
Combined with Lemma \ref{l.s_exists_s} this
implies a relation $S\subset B(x,2^\kappa \lvert x-y\rvert)$ and,
in particular, that
 $\mathbf{\beta}\le \mu(S)^{-1}$. 
The  inequality $2^{u+k}< 4\mathrm{dist}(Q,P)$, followed by $C_\lambda^\kappa=2^{d\kappa}$ and
\eqref{e.epsilon}, shows that 
$\mathbf{\alpha}\lesssim 2^{d\theta(u+m)-\eta(u+m)} \lesssim 2^{-\eta(u+m)/4}$.

Case $\ell P \ge \mathrm{dist}(Q,P)$. Choose $\kappa\in\mathbf{N}$ in such a way that
\[
2^{\kappa-1} <  \frac{c \ell S}{(\ell Q)^{\gamma} (\ell P)^{1-\gamma}}\le 2^{\kappa}\,,\qquad
c = 2^{r(1-\gamma)}\,.
\]
A useful consequence of goodness is this.
\begin{equation}\label{e.useful}
\mathrm{dist}(Q,P)> (\ell Q)^{\gamma}(\ell P)^{1-\gamma}/c\,.
\end{equation}
Lemma \ref{l.s_exists_s} and inequality \eqref{e.useful} yield $S\subset B(x,2^\kappa \lvert x-y\rvert)$, hence $\mathbf{\beta}\le \mu(S)^{-1}$.
Inequality \eqref{e.useful} also allows us to estimate
\begin{align*}
\mathbf{\alpha}\lesssim 2^{d\kappa}  \bigg(\frac{\ell Q}{\ell P}\bigg)^{\eta(1-\gamma)}
\lesssim 2^{d(m+u+\theta(u+m)) - m(d+\eta)(1-\gamma)}\lesssim 2^{-\eta(u+m)/4}\,.
\end{align*}
In the last step, we used 
the fact that $u\le 1$ and
both of the inequalities \eqref{e.epsilon}. 
\end{proof}

%

\section{Perturbations and a basic decomposition}\label{s.perturbations}

Let us denote $\mathbf{T}:=\lVert T\rVert_{L^{p_1}\to L^{p_1}}$.
We fix functions  $\widetilde{f}_j\in L^{p_j}(d\mu)$, $j=1,2$,
supported in $\mathrm{supp}(\mu)$, and satisfying
$\mathbf{T} \le 2\lvert \langle T \widetilde{f}_1,\widetilde{f}_2\rangle\rvert$
 and $\lVert \widetilde{f}_1\rVert_{p_1} = 1=\lVert \widetilde{f}_2\rVert_{p_2}$.

 For almost every pair $\{\mathcal{D}_j\,:\, j\in \{1,2\} \}$ we will
 define certain perturbations $f_j=f_j(\widetilde{f}_j,\mathcal{D}_1,\mathcal{D}_2)$ of
 functions  $\widetilde{f}_j$.
The role of these perturbations is indicated by following proposition.

\begin{proposition}\label{p.sharp} 
Under assumptions of Theorem \ref{t.main}, the following statement holds
for a fixed $t>p_1\vee p_2$. 
For every sufficiently large $r\in\mathbf{N}$ and every $\epsilon,\upsilon\in (0,1)$, 
	\begin{gather}\label{e.sharp<eta} 
		\mathbf{E}_{\omega_1}\mathbf{E}_{\omega_2}\lVert\widetilde  f_j- f_j  \rVert_{p_j}^{p_j}   \le c2^{-\gamma r/c}  \,, \qquad j=1,2\,, 
		\\  \label{e.sharpIP}
		\mathbf{E}_{\omega_1}\mathbf{E}_{\omega_2}\bigl\lvert \langle T f_1,  f_2\rangle\bigr\rvert \le  
		C(r,\upsilon,\epsilon)(1+\mathbf{T}_{\textup{loc}})
+\big(C(r,\upsilon)\epsilon^{1/t} + C(r)\upsilon^{1/t}\big)\mathbf{T}\,.
\end{gather}
Aside from the  parameters indicated, constants
$c$, $C(r,\upsilon,\epsilon)$, $C(r,\upsilon)$, and $C(r)$ are also allowed to depend
upon $n,p_1,p_2,\eta,\mu$.
\end{proposition}

Proposition \ref{p.sharp}  and
an absorption argument provide a proof of Theorem \ref{t.main}. 
Hence, we are left with proving this proposition.
During this section, we select
functions $f_j$ by using projections to good cubes, and then begin
with the analysis of the resulting  form $\bigl\lvert \langle T f_1,  f_2\rangle\bigr\rvert $.

 \subsection{Perturbations of $\widetilde{f}_j$} 
For $j\in \{1,2\}$ we denote by $Q_{j,0}$  a  cube in $\mathcal{D}_j=\mathcal{D}(\omega_j)$, containing the support $\mathrm{supp}(\mu)$ of
the measure $\mu$. Such a cube exists almost surely with respect
to $\omega_j$, \cite[Lemma 2.8]{1201.0648}. In the sequel, we will
restrict ourselves to such sequences $\omega_j$.
 Let $ \mathcal {G}_j$ be the family of all good cubes  in $\mathcal{D}_j$ that are contained in $Q_{j,0}$,
 and denote $\mathcal{G}_{j,k} = \mathcal{D}_{j,k}\cap \mathcal{G}_j$
 for  $k\in\mathbf{Z}$.
 
 We define approximates of  the functions $\widetilde{f}_j$ to be
 the following perturbations,
  \begin{equation*}
  	  f _{j} := \langle \widetilde f _{j} \rangle_{Q_{j,0}} \mathbf 1_{Q_{j,0}} 
  	  + \sum_{\substack{Q \in \mathcal G_j }} \Delta_Q \widetilde f_j\,,\qquad j=1,2 \,. 
\end{equation*}
Recall the fact that the support of $\mu$ is contained in $Q_{j,0}$. Therefore
$\Delta_Q \widetilde{f}_j = 0$ almost everywhere w.r.t. $\mu$ if $Q\in\mathcal{D}_j$ is not contained in $Q_{j,0}$.
Hence, in the view of Proposition~\ref{p.bad}, 
we have
\begin{equation}\label{e.f_diff_small}
\mathbf{E}_{\omega_1}\mathbf{E}_{\omega_2} \lVert \widetilde{f}_j - f_j\rVert_{p_j}^{p_j} 
=\mathbf{E}_{\omega_1}\mathbf{E}_{\omega_2}\lVert P_{\textup{j,bad}} \widetilde{f}_j\rVert_{p_j}^{p_j}
\le
c2^{-\gamma r/c}\,.
\end{equation}
This is inequality \eqref{e.sharp<eta}.

\subsection{Decomposition of the bilinear form}\label{s.decomposition}

During the course of the remaining sections,
we prove inequality \eqref{e.sharpIP}, which then
completes the proof of  Theorem \ref{t.main}.

By using the facts that $f_j = f_j\mathbf{1}_{Q_{j,0}}$  and $\Delta_R f_j = 0$ if $R\subset Q_{j,0}$ is a bad cube, we easily
find that an expansion of the bilinear form is
\begin{equation}\label{expansion}
	\langle T f_1,f_2 \rangle = \langle T\mathbf{E}_{Q_{1,0}} f_1, f_2\rangle 
	+ \langle T\sum_{P\in\mathcal{G}_1} \Delta_P f_1,\mathbf{E}_{Q_{2,0}} f_2\rangle
	+ \sum_{(P,Q)\in\mathcal{G}_1\times \mathcal{G}_2} \langle T \Delta_P f_1,\Delta_Q f_2\rangle.
\end{equation}
Using the assumptions and inequality \eqref{e.classical_martingale}, 
it is straightforward to verify that
\begin{align*}
|\langle T\mathbf{E}_{Q_{1,0}} f_1,f_2\rangle| +
|\langle T\sum_{P\in\mathcal{G}_1} \Delta_P f_1,\mathbf{E}_{Q_{2,0}} f_2\rangle|
\lesssim \mathbf{T}_{\textup{loc}}\lVert f_1\rVert_{p_1} \rVert f_2\rVert_{p_2}\lesssim \mathbf{T}_{\textup{loc}}\,.
\end{align*}
The last term in the right hand side
of \eqref{expansion} remains. This main term
is further split into  dual triangular sums, one
of which is the sum over 
$ (P,Q) \in \mathcal G_1 \times \mathcal G_2$ such that $ \ell P\ge \ell Q$. 
This sum will be our main point of interest, and we only remark
that the dual triangular sum, associated with cubes $\ell P<\ell Q$, is estimated in a similar manner.

The family $\{(P,Q)\in \mathcal{G}_1\times\mathcal{G}_2\,:\, \ell P\ge \ell Q\}$ is partitioned into 
three  subfamilies:
\begin{gather*}
	\mathcal P _{\textup{inside}} := 
	\{ (P,Q) \in \mathcal G_1 \times \mathcal G_2 
	\;:\;  Q \subset P\text{ and }2^r \ell Q < \ell P\}\,;
	\\
	\mathcal P _{\textup{separated}} := 
	\{ (P,Q) \in \mathcal G_1 \times \mathcal G_2 
	\;:\; \ell Q\le \ell P\text{ and } \ell Q\le  \mathrm{dist}(Q,P)\}\,;
	\\
	\mathcal P _{\textup{nearby}} := 
	\{ (P,Q) \in \mathcal G_1 \times \mathcal G_2  \;:\; 2 ^{-r} \ell P \le \ell Q \le \ell P
	\text{ and }\mathrm{dist}(Q,P)<\ell Q  \}\,.
\end{gather*}
The fact that this is a partition relies on the goodness of $Q$. We refer to \cite[Section 13]{1201.0648}
for further details.
The sums over these collections of cubes are handled separately.
Let us denote 
\[
\mathbf{B}_{\star} (f_1, f_2) = \sum_{(P,Q)\in\mathcal{P}_\star} \langle T\Delta_P f_1, \Delta_Q f_2 \rangle\,,
\qquad \star\in \{ \textup{inside},  \textup{separated}, \textup{nearby}\}\,.
\]
The analysis of the (most difficult) inside term is performed within
sections
\ref{s.inside_core} and \ref{s.remaining}.
It relies on a corona decomposition, and the
associated stopping tree is first constructed in Section \ref{s.corona}.
The separated term is analysed in a standard manner in Section \ref{s.separated}.
Finally, throughout sections \ref{s.nearby1}--\ref{s.nearby3}, we treat the nearby terms
via surgery.

 \section{A stopping tree construction}\label{s.corona}
 
 A stopping tree construction is used in the analysis of the inside-term.
 
 \smallskip

For $j\in \{1,2\}$, let us
define a stopping tree
$\mathcal S_j\subset \mathcal{D}_j$ and
a  function $ \sigma_j \;:\; \mathcal S_j \mapsto \mathbf{R} _+ $ as follows.
Take the maximal good $\mathcal{D}_j$-cubes  
$Q\subset Q_{j,0}$ in $\mathcal S_j$, and define  $ \sigma_j (Q) := \langle \lvert  f_j\rvert  \rangle _{Q} $ for these maximal cubes.
At inductive stage, if $ S\in \mathcal S_j$ 
is a minimal cube, we consider
the maximal $\mathcal{D}_j$-cubes  $ Q\subsetneq S$ subject to 
both of the conditions (1)--(2):
\begin{itemize}
\item[(1)]
 $ \langle \lvert  f_j\rvert  \rangle _{Q} > 4 \sigma_j (S)$;
 \item[(2)]
Either  $Q$ or $\pi_j Q$ is a good cube.
 \end{itemize}
We
add these cubes $Q$ to the stopping tree  $ \mathcal S _{j}$, and 
 define $ \sigma_j (Q):= \langle\lvert  f_j\rvert  \rangle _{Q}  $ for each of them.
%
%
%
%

\begin{remark}\label{r.sigma_estimate}
Condition (2), imposed
in the construction of stopping trees,
will be useful to us in many occasions.
A minor side effect is that we can rely on inequality
$\langle \lvert f_j\rvert\rangle_{Q} \le 4\sigma_1(\pi_{\mathcal{S}_j}Q)$
 for a $\mathcal{D}_j$-dyadic cube $Q\subsetneq Q_{j,0}$ only if
either $Q$ or $\pi_j Q$ is good.  But this is, in fact, all we need.
\end{remark}


\begin{remark}\label{r.sparse}
By construction  $\mathcal{S}_j$ is a `sparse family of cubes', i.e.,
\begin{equation}\label{e.sparse}
\sum_{S'\in\textup{ch}_{\mathcal{S}_j} (S)} \mu(S') \le 4^{-1} \mu(S)\,,\qquad S\in\mathcal{S}_j\,,\quad j\in \{1,2\}\,.
\end{equation}
In particular, family $\mathcal{S}_j$ satisfies a `Carleson condition':
$\sum_{S'\in\mathcal{S}_j:S'\subset S} \mu(S')\lesssim \mu(S)$
if $S\in\mathcal{S}_j$.
\end{remark}

 \subsection{Quasi-orthogonality}
The following is a key inequality,
 \begin{equation}\label{e.quasi}
 	 \sum_{S\in \mathcal S_j} \sigma _{j} (S) ^{p_j} \mu( S) \lesssim \lVert  f_j\rVert_{p_j} ^{p_j} \lesssim  1\,,\qquad j\in \{1,2\}\,.
\end{equation}

\begin{proof}[Proof of \eqref{e.quasi}]
We apply a
dyadic maximal function:
$M_{j,\mu} f_j(x)=\sup_{x\in Q\in\mathcal{D}_j} \langle \lvert f_j\rvert\rangle_Q$.
For $ S\in \mathcal S _{j}$, we let $ E_{S}$ be the 
set $ S$ minus all the $ \mathcal S _{j}$-children of $ S$.  By inequality \eqref{e.sparse}, 
 $ \mu(  E_{S})  \ge \tfrac 34 
\mu(  S)  $, and the sets $ E_{S}$ are 
pairwise disjoint by definition.  
 Hence,
\begin{align*}
\sum_{S\in \mathcal S _{j}} \sigma_{j} (S) ^{p_j} \mu(  S) 
&\le \tfrac 43 \sum_{S\in \mathcal S _{j}} \langle \lvert f_j\rvert\rangle_{S} ^{p_j} \mu(  E_{S}) 
\\& \lesssim \sum_{S\in \mathcal S _{j}} \int _{E_{S}} (M_{j,\mu} f_j ) ^{p_j} \, d\mu \le \int_{\mathbf{R}^n}(M_{j,\mu} f_j)^{p_j}\,d\mu\,.  
\end{align*}
Thus, the first inequality in \eqref{e.quasi} follows
from the fact that $M_{j,\mu}$ is bounded on $L^{p_j}$. The second
inequality is a consequence of the martingale transform inequality \eqref{e.classical_martingale}.
\end{proof}

\subsection{Martingale projections}\label{s.martingale_proj}
For  $S\in\mathcal{S}_j$
and $\phi\in L^1_{\textup{loc}}$, we  define $P_{j,S} \phi = \sum_{Q\in\mathcal{D}_j\,:\,\pi_{\mathcal{S}_j}Q =S} \Delta_Q \phi$. 
By orthogonality of martingale differences and  inequality \eqref{e.classical_martingale},
for all $S\in\mathcal{S}_j$ and all sequences of constants satisfying
$\sup_{Q\in\mathcal{G}_j}
\lvert \varepsilon_{Q}\rvert \le 1$,
\begin{equation}\label{e.mart_ineq}
\bigg\lVert \sum_{Q\in\mathcal{G}_j\,:\,\pi_{\mathcal{S}_j}Q =S} \varepsilon_{Q} \Delta_Q f_j\bigg\rVert_{p_j}^{p_j}\lesssim \lVert P_{j,S} f_j\rVert_{p_j}^{p_j}\,.
\qquad 
\end{equation}
Of fundamental importance is the following inequality, which 
\emph{does not hold} for general families of 
orthogonal martingale projections, 
in the case of $1 <p_j <2$.   
\begin{equation}\label{e.proj_ineq}
\sum_{S\in\mathcal{S}_j} \lVert P_{j,S} f_j\rVert_{p_j}^{p_j}
 \lesssim 1\,.
\end{equation}

\begin{proof}[Proof of \eqref{e.proj_ineq}]
Let us write
\begin{equation}\label{e.easy_est}
\sum_{S\in\mathcal{S}_j} \lVert P_{j,S} f_j\rVert_{p_j}^{p_j}
= \sum_{S\in\mathcal{S}_j}\lVert \mathbf{1}_{S\setminus E_{S}} P_{j,S} f_j \rVert_{p_j}^{p_j}
+
\sum_{S\in\mathcal{S}_j}\lVert \mathbf{1}_{E_{S}}P_{j,S} f_j \rVert_{p_j}^{p_j}
\,,
\end{equation}
where $E_S$ denotes the set $S\setminus \bigcup_{S'\in\textup{ch}_{\mathcal{S}_j}(S)} S'$. We estimate the two terms separately.
First,
\begin{align*}
\lvert  \mathbf{1}_{S\setminus E_{S}} P_{j,S} f_j\rvert
 = \bigg|\sum_{S'\in\textup{ch}_{\mathcal{S}_j} (S)} \mathbf{1}_{S'}
 \{ \langle f_j\rangle_{S'} - \langle f_j\rangle_{S}\}\bigg\rvert
 \lesssim \sum_{S'\in\textup{ch}_{\mathcal{S}_j} (S)} \mathbf{1}_{S'} \sigma_j(S')\,.
\end{align*}
Since the family $\textup{ch}_{\mathcal{S}_j}(S)$ is disjoint, the
upper bound for the first term in RHS\eqref{e.easy_est} follows from inequality \eqref{e.quasi}.

By \eqref{e.quasi} it remains to show that
$\lvert\mathbf{1}_{E_S} P_{j,S} f_j\rvert \lesssim \mathbf{1}_{E_S} \sigma_j(S)$ 
almost everywhere.
 We  restrict
ourselves to points in which $\lim_{k\to -\infty}\mathbf{E}_{j,k} f_j(x)=f_j(x)$, hence
$\lvert \mathbf{1}_{E_S}(x) P_{j,S} f_j(x) \rvert = \lvert \mathbf{1}_{E_S}(x) \{f_j(x) - \langle f_j\rangle_{S}\}\rvert$.
Observe  that  $\lvert \langle f_j\rangle_{S}\rvert \le \sigma_j(S)$.
Now, there are three cases (1)--(3) for $x\in E_S$ as above:

(1) If there are no good $\mathcal{D}_j$-cubes inside $S$ containing $x$,
 we  have $P_{j,S} f_j(x)=0$ by definitions.

(2) There is a minimal good $\mathcal{D}_j$-cube $Q\subset S$ containing $x$, in
which case we let $Q_x\in\textup{ch}(Q)$ be the child
containing $x$. If $R\subset Q_x$ is a
$\mathcal{D}_j$-cube containing $x$, we easily
find that $\langle f_j\rangle_{R} =\langle f_j\rangle_{Q_x}$.
Thus, by martingale convergence, 
\[
\lvert f_j(x)\rvert = \lim_{\ell R \to 0} \lvert \langle f_j\rangle_{R}\rvert
=\lvert \langle f_j\rangle_{Q_x}\rvert
\le 4\sigma_j(\pi_{\mathcal{S}_j} Q_x)=4\sigma_j(S)\,.
\]
In the penultimate step above, we used 
Remark \ref{r.sigma_estimate} and the fact that $\pi_j Q_x=Q$ is good.
And, in the last step, we used the fact that $x\in E_S$.

(3) There are arbitrarily small good $\mathcal{D}_j$-cubes $Q\subset S$ containing $x$.
Hence,
\begin{align*}
\lvert f_j(x)\rvert 
=\lim_{\ell Q\to 0} \lvert \langle f_j\rangle_{Q}\rvert
\le \sup \{  \langle \lvert f_j\rvert \rangle_{Q}\,:\, x\in Q\subset S\}
\le 4\sigma_j(S)\,.
\end{align*}
The limit and supremum above are restricted to good $\mathcal{D}_j$-cubes 
satisfying $x\in Q\subset S$.
\end{proof}

\subsection{Family $\mathcal{L}_2(S)$ and its layers}\label{s.layers}
This construction is needed as we study the case of $ p_j\neq 2$, and in particular it will allow us to more freely use the inequality \eqref{e.proj_ineq}. 

For $S\in\mathcal{S}_1$ let us define
 $\mathcal{L}_2(S)\subset \mathcal{S}_2$ to be the family of cubes
 of the form
 $R=\pi_{\mathcal{S}_2}Q$, 
where $Q\in\mathcal{G}_2$ satisfies $(P,Q)\in\mathcal{P}_{\textup{inside}}$ for
some cube $P\in\mathcal{G}_1$ with $\pi_{\mathcal{S}_1} P_Q = S$. Here $P_Q$
stands for the child of $P$ containing $Q$; it exists
by goodness of $Q$.

Lemma \ref{l.layers} records the
observation that there are at most $2(r+1)$ layers in
 $\mathcal{L}_2(S)$ which contain cubes $R$ such that $R\not\subset S$.
To be more precise, let $\mathcal{L}_2^k(S)$
denote the layer $k\ge 0$ cubes in $\mathcal{L}_2(S)$, i.e.,  the
cubes $R$ in this family for which $\pi_{\mathcal{L}_2(S)}^k R$ is a maximal cube in $\mathcal{L}_2(S)$. 

\begin{lemma}\label{l.layers}
Suppose that $S\in\mathcal{S}_1$ and $R\in \mathcal{L}_2^k(S)$ with $k\ge 2(r+1)$.
Then $R\subset S$.
\end{lemma}

\begin{proof} We first claim that, if $R\in \mathcal{L}_2^k(S)$  with $k\ge 1$, then
\begin{equation}\label{e.basic_estimate}
2^{k-1} \ell R \le 2^r\ell S\,.
\end{equation}
The lemma is a  consequence of inequality \eqref{e.basic_estimate}.
Indeed, 
if $k\ge 2(r+1)$, 
we then have $2^{r+1}\ell R \le \ell S$ and
 $S\cap R\not=\emptyset$. It remains to recall 
that either $R$ or $\pi R$ is a good (by construction).

Let us then prove inequality \eqref{e.basic_estimate}.
Clearly, it suffices to
verify the case of $k=1$.
Suppose that
$R\subsetneq R_{0}$ is a cube
in the first layer, and $R_{0}\in\mathcal{L}_2^0(S)$ is maximal.
Then, by definition, there are cubes $Q,Q'\in\mathcal{G}_2$ such that
$Q\cup Q'\subset S$, $Q\subset R$ and 
$Q'\cap( R_{0}\setminus R)\not=\emptyset$.
From these facts it easily follows that $S\cap R\not=\emptyset$ and $\mathrm{dist}(S,\partial R)=0$.
Since either $S$ or $\pi_1 S$ is a good cube, 
$\ell R\le 2^r\ell S$.
\end{proof}

\subsection{Further inequalities}
The reader may omit this technical section
for the time being.
The following important inequality parallels \eqref{e.proj_ineq};
recall definition of $\mathcal{L}_2(S)$ in Section \ref{s.layers}:---for all 
sequences of constants satisfying $\sup_{\mathcal{G}_2\times\mathcal{S}_1} \lvert \varepsilon_{Q,S} \rvert \le 1$,
\begin{equation}\label{e.proj_ineqII}
\sum_{S\in\mathcal{S}_1} 
\sum_{S'\in\textup{ch}_{\mathcal{S}_1}^t(S)}
\sum_{\substack{ R\in\mathcal{L}_2(S) \\ R\not\subset S}}
 \bigg
\lVert \sum_{\substack{ Q\in\mathcal{G}_2 : \pi_{\mathcal{S}_2} Q = R \\\pi_{\mathcal{S}_1} Q = S'}} \varepsilon_{Q,S} \Delta_Q f_2\bigg\rVert_{p_2}^{p_2}
\lesssim 1\,,\quad \text{ if }
t\ge 0\,.
\end{equation}
\begin{proof}[Proof of inequality \eqref{e.proj_ineqII}]
Let us fix $t\ge 0$. First,
by martingale  transform inequality \eqref{e.classical_martingale} and 
orthogonality of martingale differences, 
we can assume that $\varepsilon_{Q,S}=1$ for all $Q$ and $S$.
By Lemma \ref{l.diff_est}, we obtain an upper bound
\begin{equation}\label{e.this_est}
\sum_{R\in\mathcal{S}_2}  \sigma_2(R)^{p_2}
\sum_{S\in \mathcal{S}_1;R\not\subset S} \mu(S\cap R) 
+ 
\sum_{R\in\mathcal{S}_2}
\sum_{R'\in\textup{ch}_{\mathcal{S}_2}(R)} \sum_{S\in\mathcal{S}_1} 
\sum_{\substack{S'\in\textup{ch}_{\mathcal{S}_1}^t(S)\\ \pi_{\mathcal{S}_1}(\pi_2 R')=S'}}
\sigma_2(R')^{p_2}\mu(R') \,.
\end{equation}
By inequality \eqref{e.quasi}, the second term is bounded by $\lesssim 1$;
Indeed, for a fixed $R'$ there is at most one pair of cubes $S,S'$ such that $\pi_{\mathcal{S}_1}(\pi_2 R')=S'$.

Concerning the first term in \eqref{e.this_est}, we observe that $\sum_{S\in \mathcal{S}_1;R\not\subset S} \mu(S\cap R) 
\lesssim \mu(R)$ and then apply inequality \eqref{e.quasi}. Mentioned observation is reached by splitting the series in two parts, depending
if $S\subset R$ or not;  The  series with $S\subset R$ is estimated by the Carleson condition,
Remark \ref{r.sparse}.
The second series, in which $S\not\subset R$, is estimated
by using the fact that $2^{-r}\ell S\le \ell R\le 2^r \ell S$ if
$S\cap R\not=\emptyset$, $S\not\subset R$, and $R\not\subset S$. Indeed, there are at most
$c(n,r)$ such cubes $S\in\mathcal{S}_1$ for a fixed $R\in\mathcal{S}_2$.
\end{proof}

\begin{lemma}\label{l.diff_est}
Let us fix
$S\in\mathcal{S}_1$, $S'\in \textup{ch}_{\mathcal{S}_1}^t(S)$, and $R\in\mathcal{L}_2(S)$ such that
$R\not\subset S$. 
Then
\begin{equation}\label{e.desired}
\bigg
\lVert \sum_{\substack{ Q\in\mathcal{G}_2 : \pi_{\mathcal{S}_2} Q = R \\\pi_{\mathcal{S}_1} Q = S'}} \Delta_Q f_2\bigg\rVert_{p_2}^{p_2}
\lesssim 
\mu(S'\cap R) \sigma_2(R)^{p_2} +
\sum_{R'\in\textup{ch}_{\mathcal{S}_2}(R)}
\mathbf{1}_{\pi_{\mathcal{S}_1} (\pi_2 R')=S'}  \mu(R') \sigma_2(R')^{p_2}\,.
\end{equation}
\end{lemma}

\begin{proof}
Let $E_R$ be the set $R$ take away all the $\mathcal{S}_2$-children of $R$.
Then, LHS\eqref{e.desired} is bounded by 
\[
\mathbf{A} + \mathbf{B}= \bigg
\lVert 
\mathbf{1}_{R\setminus E_R}\sum_{\substack{ Q\in\mathcal{G}_2 : \pi_{\mathcal{S}_2} Q = R \\\pi_{\mathcal{S}_1} Q = S'}} \Delta_Q f_2\bigg\rVert_{p_2}^{p_2}
+ \bigg
\lVert \mathbf{1}_{E_R\cap S'}\sum_{\substack{ Q\in\mathcal{G}_2 : \pi_{\mathcal{S}_2} Q = R \\\pi_{\mathcal{S}_1} Q = S'}} \Delta_Q f_2\bigg\rVert_{p_2}^{p_2}\,.
\]
In order to estimate $\mathbf{A}$, let us consider a child
$R'\in\textup{ch}_{\mathcal{S}_2}(R)$ for which there
are cubes $Q^+=Q^+(R')$ and
$Q^-=Q^-(R')$:---these are the maximal and minimal cubes, respectively, subject
to conditions $Q\in\mathcal{G}_2$, $\pi_{\mathcal{S}_2} Q = R$,
$\pi_{\mathcal{S}_1} Q=S'$, and $R'\subsetneq Q$. Then
\begin{equation}\label{e.r'_estimate}
\begin{split}
\bigg\lvert \mathbf{1}_{R'}\sum_{\substack{ Q\in\mathcal{G}_2 : \pi_{\mathcal{S}_2} Q = R \\\pi_{\mathcal{S}_1} Q = S'}} \Delta_Q f_2
\bigg\rvert
 &= \mathbf{1}_{R'\cap S'}\cdot \lvert
\{\langle f_2\rangle_{Q^{-}_{R'}}  - \langle f_2\rangle_{Q^+}\}\rvert\\
&\lesssim  \begin{cases}
\mathbf{1}_{\pi_{\mathcal{S}_1}(\pi_2 R')=S'}\mathbf{1}_{R'} \sigma_2(R')\,,\quad &\text{ if } Q^{-}_{R'}=R'\,;\\
\mathbf{1}_{S'\cap R'} \sigma_2(R)\,,\quad &\text{otherwise}\,.
\end{cases}
\end{split}
\end{equation}
We used the facts
that $\Delta_Q f_2=0$ if $Q\subset Q_{2,0}$ is  bad,
and that $\pi_{\mathcal{S}_2} Q^{-}_{R'}=R$ if $Q^{-}_{R'}\not=R'$.
Writing $R\setminus E_R=\sum_{R'\in\textup{ch}_{\mathcal{S}_2}(R)} \mathbf{1}_{R'}$ and using disjointness of these children yields
\begin{equation}\label{e.a_estimate}
\mathbf{A} \lesssim 
\mu(S'\cap R) \sigma_2(R)^{p_2} +
\sum_{R'\in\textup{ch}_{\mathcal{S}^2}(R)}
\mathbf{1}_{\pi_{\mathcal{S}_1} (\pi_2 R')=S'}  \mu(R') \sigma_2(R')^{p_2}\,.
\end{equation}

We turn to term $\mathbf{B}$;
we will implicitly use the fact that $\Delta_Q f_2=0$ if $Q\subset Q_{2,0}$ is a bad $\mathcal{D}_2$-cube.
Let us fix 
a point $x\in E_R\cap S'$ such that $\lim_{k\to -\infty}\mathbf{E}_{2,k} f_2(x)=f_2(x)$, and there is a maximal cube $Q^+\in\mathcal{G}_2$, subject to conditions
$x\in Q\in\mathcal{G}_2$, $\pi_{\mathcal{S}_2} Q=R$, $\pi_{\mathcal{S}_1}Q=S'$. Then,
\begin{equation}\label{e.x_sums}
\bigg\lvert \sum_{\substack{ Q\in\mathcal{G}_2 : \pi_{\mathcal{S}_2} Q = R \\\pi_{\mathcal{S}_1} Q = S'}} \Delta_Q f_2 (x)
\bigg\rvert
= 
\bigg\lvert\sum_{ \substack{Q\in\mathcal{G}_2:\pi_{\mathcal{S}_1} Q = S'\\Q\subset Q^+}} \Delta_Q f_2(x)\bigg\rvert\,.
\end{equation}
We aim to verify that $RHS\eqref{e.x_sums}\lesssim \sigma_2(R)$.
This allows us to conclude that
$\mathbf{B}\lesssim \mu(S'\cap R)\sigma_2(R)^{p_2}$.
There are two cases. First, $\pi_{\mathcal{S}_1}Q = S'$ for
all cubes $x\in Q\in\mathcal{G}_2$ with $Q\subset Q^+$;
In this case, we proceed as in the proof of \eqref{e.proj_ineq}  in order to see that
$ RHS\eqref{e.x_sums} = \lvert f_2(x) - \langle f_2 \rangle_{Q^+}\rvert \lesssim \sigma_2(R)$.
Second, there is a minimal cube $Q^-$ subject to
conditions $\pi_{\mathcal{S}_1} Q = S'$, $x\in Q\in\mathcal{G}_2$,
$ Q\subset Q^+$. In this case, we find that
$RHS\eqref{e.x_sums} =\lvert \langle f_2\rangle_{Q^-_x} - \langle f_2\rangle_{Q^+}\rvert\lesssim \sigma_2(R)$,
 where $Q^-_x$ denotes the child of $Q^-$, containing $x$.
In the  last step, we used the fact that $x\in E_R$ so 
that
$\pi_{\mathcal{S}_2}(Q^-_x) = R$.
\end{proof}

\section{The Inside-Paraproduct Term}\label{s.inside_core}

First we decompose the inside term
$\mathbf{B}_{\textup{inside}}(f_1,f_2)$,
 associated with the indexing
cubes $\mathcal{P}_{\textup{inside}}$. There will be three terms
labelled as:
`paraproduct', `stopping', and `error'. The
`paraproduct' term is treated in this section.
The other ones are treated in Section \ref{s.remaining}.

The conditions for $ (P,Q) \in	\mathcal P _{\textup{inside}} $ are:
$P\in\mathcal{G}_1$, $Q\in\mathcal{G}_2$, 
$ Q\subset P$, 
and $ 2 ^{r} \ell Q < \ell P$.
These are abbreviated to $ Q\Subset P$. 
The child of $P$ containing $Q$ is denoted by $P_Q$; it exists by goodness of $Q$.
For $(P,Q)\in\mathcal{P}_{\textup{inside}}$ we write
\begin{align*}
{\Delta}_P f_1 &=  \langle {\Delta}_P f_1\rangle_{P_Q} \mathbf{1}_{\pi_{\mathcal{S}_1}P_Q} - 
\langle {\Delta}_P f_1\rangle_{P_Q} \mathbf{1}_{\pi_{\mathcal{S}_1}P_Q\setminus P_Q} + {\Delta}_P f_1 \cdot \mathbf{1}_{P\setminus P_Q}\,.
\end{align*}
This equation 
is valid pointwise $\mu$-almost everywhere (everywhere if $\mu(P_Q)\not=0$), and 
it  yields the following expansion, respectively,
\begin{equation}\label{e.inside_decomposition}
\begin{split}
\mathbf{B}_{\textup{inside}}(f_1,f_2)&=\sum_{(P,Q)\in\mathcal{P}_{\textup{inside}}} 
\langle T{\Delta}_P f_1, \Delta_Q f_2\rangle
\\&= {\mathbf{B}}^{\textup{para}}(f_1,f_2)
- {\mathbf{B}}^{\textup{stop}}(f_1,f_2) + {\mathbf{B}}^{\textup{error}}(f_1,f_2)\,,
\end{split}
\end{equation}
Hence, e.g., ${\mathbf{B}}^{\textup{para}}(f_1,f_2)=\sum_{(P,Q)\in\mathcal{P}_{\textup{inside}}} 
\langle \Delta_P f_1\rangle_{P_Q} \langle T\mathbf{1}_{\pi_{\mathcal{S}_1}P_Q}, \Delta_Q f_2\rangle$.
%
The main result in this section is the following estimate for the paraproduct term.

\begin{proposition}\label{p.para_estimate}
We have inequality
$\big\lvert {\mathbf{B}}^{\textup{para}}(f_1,f_2)\big\lvert  \lesssim 1+\mathbf{T}_{\textup{loc}}$.
\end{proposition}

The remainder of this section is dedicated to
the proof of this proposition, and
the main focus will be on auxiliary inequalities
\eqref{e.subset_ineq}
and \eqref{e.subset_ineqII}.
Let us first examine how these
 inequalities are used to prove
Proposition \ref{p.para_estimate}.
First, for $S\in\mathcal{S}_1$, recall definition of $\mathcal{L}_2(S)$ given in Section \ref{s.layers}.
We define
\begin{align*}
&{\mathbf{B}}^{\textup{para}}_{S,\not\subset}(f_1,f_2) + {\mathbf{B}}^{\textup{para}}_{S,\subset}(f_1,f_2)\\&=\bigg\{ \sum_{\substack {R\in \mathcal{L}_2(S) \\ R\not\subset S }} + \sum_{\substack {R\in \mathcal{L}_2(S) \\R\subset S }}\bigg\}
\sum_{Q\in\mathcal{G}_2:\pi_{\mathcal{S}_2}Q=R} \sum_{\substack{P\in\mathcal{G}_1: \pi_{\mathcal{S}_1}P_Q = S \\ Q\Subset P}}
\langle {\Delta}_P f_1\rangle_{P_Q}\langle T\mathbf{1}_S,\Delta_Q f_2\rangle\,.
\end{align*}
Then, by the auxiliary inequalities mentioned above,
\begin{align*}
\rvert {\mathbf{B}}^{\textup{para}}(f_1,f_2)\rvert & \le
\bigg\lvert \sum_{S\in\mathcal{S}_1}  {\mathbf{B}}^{\textup{para}}_{S,\not\subset}(f_1,f_2)  \bigg\rvert
+\bigg\lvert \sum_{S\in\mathcal{S}_1}  {\mathbf{B}}^{\textup{para}}_{S,\subset}(f_1,f_2)  \bigg\rvert
\lesssim 1+\mathbf{T}_{\textup{loc}}\,.
\end{align*}
This concludes the proof of Proposition
\ref{p.para_estimate}, assuming the auxiliary inequalities.


\subsection{A telescoping identity}

For  fixed $Q\in\mathcal{G}_2$ and $S\in\mathcal{S}_1$, let us define
a constant $\varepsilon_{Q,S}$ by
\begin{equation}\label{epsilon_def}
\varepsilon_{Q,S}\sigma_1(S) = \sum_{\substack{ P\in\mathcal{G}_1:\pi_{\mathcal{S}_1} P_Q = S \\ Q\Subset P}}
\langle \Delta_P f_1\rangle_{P_Q}\,.
\end{equation}
 It is important to use the condition $\pi_{\mathcal{S}_1}P_Q=S$
 instead of $\pi_{\mathcal{S}_1}P=S$. Otherwise,
the following important lemma might fail, as
the measure $\mu$ need not be doubling.

\begin{lemma}\label{l.epsilon_bounded} For
$Q\in\mathcal{G}_2$ and $S\in\mathcal{S}_1$, we have $\lvert \varepsilon_{Q,S}\rvert \lesssim 1$.
\end{lemma}

\begin{proof}
Recall our convention that $\langle \Delta_P f_1\rangle_{P_Q}=0$ if $\mu(P_Q)=0$.
Consider the minimal and maximal dyadic cubes: $P^{-}$ and $P^+$,
subject to conditions 
$P\in\mathcal{G}_1$, $\pi_{\mathcal{S}_1} P_Q = S$,
$Q\Subset P$, and $\mu(P_Q)\not=0$.
If such cubes do not exist, we are done.
Otherwise, we claim that
\begin{equation}\label{e.telescope}
\varepsilon_{Q,S} \sigma_1(S) = \langle f_1\rangle_{P_Q^{-}}
-\langle f_1\rangle_{P^+}\,.
\end{equation}
By using equation \eqref{e.telescope} and the construction of the stopping tree,
we find that $\lvert \varepsilon_{Q,S}\rvert\le 8$. 

It remains to prove equation \eqref{e.telescope}. Suppose that
 $P\in\mathcal{G}_1$ is such that
$\pi_{\mathcal{S}_1} P_Q = S$, $Q\Subset P$, and
$\mu(P_Q)\not=0$.
Then $P^-\subset P\subset P^+$. 
By this observation,
\begin{equation}\label{e.telescopeII}
\sum_{\substack{ P\in\mathcal{G}_1:\pi_{\mathcal{S}_1} P_Q = S \\ 
Q\Subset P:\mu(P_Q)\not=0}}
 {\Delta}_P f_1\cdot \mathbf{1}_{P_Q^-} =
 \sum_{\substack{ P\in\mathcal{G}_1\\ P^-\subset P\subset P^+}}
 \mathbf{1}_{\pi_{\mathcal{S}_1}P_Q  = S} \mathbf{1}_{Q\Subset P}
\mathbf{1}_{\mu(P_Q)\not=0} \cdot
{\Delta}_P f_1\cdot \mathbf{1}_{P_Q^-}\,.
\end{equation}
Observe that $ \mathbf{1}_{\pi_{\mathcal{S}_1}P_Q  = S} \mathbf{1}_{Q\Subset P}\mathbf{1}_{\mu(P_Q)\not=0}=1$
inside the summation. Also,
$\Delta_P f_1= 0$ if $P$ is a bad cube with $P^- \subset P\subset P^+$. Thus,
by adding the zero contribution from the bad cubes in a formal manner, we obtain
a telescoping identity: 
$LHS\eqref{e.telescopeII} = \{ \langle f_1\rangle_{P^-_Q} - \langle f_1\rangle_{P^+}  \}\mathbf{1}_{P_Q^-}$.
The equation \eqref{e.telescope} follows from this: first, we restrict ourselves
to cubes $P$ with $\mu(P_Q)\not=0$ in the series defining $\varepsilon_{Q,S}$. Then, we
replace the $P_Q$ averages by $P_Q^{-}$ averages inside the
summation; observe that $P_Q^-\subset P_Q$ and $\mu(P_Q^-)\not=0$.
Finally, we exchange the order of summation and the brackets,
and apply the obtained telescoping identity. \end{proof}

\subsection{Summation involving cubes $R\not\subset S$}
Our aim in this section is to prove an inequality,
\begin{equation}\label{e.subset_ineq}
\bigg\lvert \sum_{S\in\mathcal{S}_1}  {\mathbf{B}}^{\textup{para}}_{S,\not\subset}(f_1,f_2)  \bigg\rvert
\lesssim 1+\mathbf{T}_{\textup{loc}}\,.
\end{equation}
Let us express the series defining $ {\mathbf{B}}^{\textup{para}}_{S,\not\subset}(f_1,f_2)$ in a convenient manner. 
For this purpose, observe that $Q\subset P_Q\subset S$ for any cube $Q$ in the
series defining ${\mathbf{B}}^{\textup{para}}_{S,\not\subset}(f_1,f_2)$. In particular,
$\pi_{\mathcal{S}_1}Q\subset S$. Thus, by organising
the $Q$-summation in terms of
their $\mathcal{S}_1$-parents and  
defining $\varepsilon_{Q,S}$ as the solution to equation \eqref{epsilon_def}, we find that
\begin{equation}\label{e.complex}
{\mathbf{B}}^{\textup{para}}_{S,\not\subset}(f_1,f_2)  = \sigma_1(S)
\sum_{t\ge 0} \sum_{\substack {R\in \mathcal{L}_2(S) \\R\not\subset S }}
\sum_{S'\in \textup{ch}_{\mathcal{S}_1}^t(S)}
\sum_{\substack{Q\in\mathcal{G}_2:\pi_{\mathcal{S}_2}Q=R\\ \pi_{\mathcal{S}_1}Q=S'}} 
\langle \mathbf{1}_R\mathbf{1}_{S'}\{T\mathbf{1}_S - \mathbf{\tau}_{t,S'} \},\varepsilon_{Q,S} \Delta_Q f_2\rangle\,.
\end{equation}
By using the fact that  $\Delta_Q f_2$ has mean zero, we have also subtracted off the constants 
\[
\tau_{t,S'} = \begin{cases} 0\,,\quad &\text{ if } t\in \{0,\ldots, 2r+1\}\,;\\
T\mathbf{1}_{S\setminus \pi_{\mathcal{S}_1}^{\lfloor t/2\rfloor} S'}(x_{S'})\,, &\text{ otherwise}\,.
\end{cases}
\]
For convenience, let us denote
\[
\mathbf{A}_{t,S}= \bigg\{
\sum_{\substack {R\in \mathcal{L}_2(S) \\R\not\subset S }}
\sum_{S'\in \textup{ch}_{\mathcal{S}_1}^t(S)} 
\lVert  \mathbf{1}_R\mathbf{1}_{S'}\{T\mathbf{1}_S - \mathbf{\tau}_{t,S'} \} \rVert_{p_1}^{p_1}
\bigg\}^{1/p_1}\,,
\]
and
\[
\mathbf{B}_{t,S} = \bigg\{ \sum_{\substack {R\in \mathcal{L}_2(S) \\R\not\subset S }}
\sum_{S'\in \textup{ch}_{\mathcal{S}_1}^t(S)} 
\bigg\lVert  \sum_{\substack{Q\in\mathcal{G}_2:\pi_{\mathcal{S}_2}Q=R\\ \pi_{\mathcal{S}_1}Q=S'}} \varepsilon_{Q,S}\Delta_Q f_2 \bigg\rVert_{p_2}^{p_2}
\bigg\}^{1/p_2}\,.
\]
The useful inequality
$\sup_{\mathcal{G}_2\times\mathcal{S}_1}\lvert \varepsilon_{Q,S}\rvert\lesssim 1$ is a consequence of Lemma \ref{l.epsilon_bounded}.
By equation \eqref{e.complex} and H\"older's inequality, combined with
 Lemma \ref{l.tss'_1},
\begin{align*}
\bigg\lvert \sum_{S\in\mathcal{S}_1}  {\mathbf{B}}^{\textup{para}}_{S,\not\subset}(f_1,f_2)  \bigg\rvert
&\lesssim 
\sum_{t\ge 0} \sum_{S\in\mathcal{S}_1}\sigma_1(S)\mathbf{A}_{t,S}\mathbf{B}_{t,S}\\
&\lesssim
(1+\mathbf{T}_{\textup{loc}})\sum_{t\ge 0} 2^{-t/p_1}
\bigg\{\sum_{S\in\mathcal{S}_1} \sigma_1(S)^{p_1} \mu(S)\bigg\}^{1/p_1}
\bigg\{ \sum_{S\in\mathcal{S}_1} \mathbf{B}_{t,S}^{p_2}\bigg\}^{1/p_2}\,.
\end{align*}
Inequality \eqref{e.subset_ineq} is  obtained by applying
inequalities 
\eqref{e.quasi} and \eqref{e.proj_ineqII},
and summing the geometric series afterwards. 

\begin{lemma}\label{l.tss'_1}
For every $t\ge 0$ and $S\in\mathcal{S}_1$, we have 
$\mathbf{A}_{t,S}\lesssim (1+\mathbf{T}_{\textup{loc}})2^{-t/p_1}\mu(S)^{1/p_1}$.
\end{lemma}

\begin{proof}
By Lemma \ref{l.layers} and the fact that 
layers $\mathcal{L}_2^k(S)$, $k\ge 0$, are comprised of disjoint cubes,
we can bound $\mathbf{A}_{t,S}^{p_1}$ by
\begin{equation}\label{e.basic_red}
\sum_{k=0}^{2r+1}
\sum_{S'\in \textup{ch}_{\mathcal{S}_1}^t(S)} 
\sum_{\substack {R\in \mathcal{L}_2^k(S) \\R\not\subset S }}
\lVert  \mathbf{1}_R\mathbf{1}_{S'}\{T\mathbf{1}_S - \mathbf{\tau}_{t,S'} \} \rVert_{p_1}^{p_1}
\lesssim 
\sum_{S'\in \textup{ch}_{\mathcal{S}_1}^t(S)} 
\lVert  \mathbf{1}_{S'}\{T\mathbf{1}_S - \mathbf{\tau}_{t,S'} \} \rVert_{p_1}^{p_1}\,.
\end{equation}
Let us first focus on the case of $t\in \{0,\ldots,2r+1\}$. 
By inequality \eqref{e.basic_red}
and the facts that cubes in $\textup{ch}_{\mathcal{S}_1}^t(S)$ are disjoint
and they are contained in $S$,
\begin{align*}
\mathbf{A}_{t,S}^{p_1}\lesssim \lVert \mathbf{1}_ST\mathbf{1}_{S}\rVert_{p_1}^{p_1} \le 
\mathbf{T}_{\textup{loc}}^{p_1}\mu(S)\lesssim 
(1+\mathbf{T}_{\textup{loc}})^{p_1}
2^{-t}\mu(S)\,.
\end{align*}
Let us then focus on the case of $t\ge 2r+2$; we begin by writing
\[
RHS\eqref {e.basic_red}=\sum_{\substack{ S''\in\textup{ch}^{\lceil t/2\rceil}_{\mathcal{S}_1}(S)} }
\sum_{\substack {S'\in\textup{ch}^t_{\mathcal{S}_1}(S) \\ \pi_{\mathcal{S}_1}^{\lfloor t/2\rfloor} S'=S''}} 
\lVert \mathbf{1}_{S'} \{T\mathbf{1}_S-T\mathbf{1}_{S\setminus S''}(x_{S'})\}\rVert_{p_1}^{p_1}\,.
\]
To conclude the proof of lemma, it suffices to first verify that for all 
$S''\in\textup{ch}^{\lceil t/2\rceil}_{\mathcal{S}_1}(S)$,
\begin{equation}\label{e.auxiliary_est}
\sum_{\substack {S'\in\textup{ch}^t_{\mathcal{S}_1}(S) \\ \pi_{\mathcal{S}_1}^{\lfloor t/2\rfloor} S'=S''}} 
\lVert \mathbf{1}_{S'} \{T\mathbf{1}_S-T\mathbf{1}_{S\setminus S''}(x_{S'})\}\rVert_{p_1}^{p_1}
\lesssim (1+\mathbf{T}_{\textup{loc}})^{p_1}\mu(S'')\,,
\end{equation}
and then inductively apply the sparseness property of $\mathcal{S}_1$, we refer to Remark \ref{r.sparse}.

In order to prove the remaining inequality \eqref{e.auxiliary_est}, we estimate
LHS\eqref{e.auxiliary_est}  by $2^{p_1-1}(\alpha+\beta)$,
\begin{align*}
\alpha + \beta=\sum_{\substack {S'\in\textup{ch}^t_{\mathcal{S}_1}(S) \\ \pi_{\mathcal{S}_1}^{\lfloor t/2\rfloor} S'=S''}} \lVert \mathbf{1}_{S'} T\mathbf{1}_{S''}\rVert_{p_1}^{p_1}
+\sum_{\substack {S'\in\textup{ch}^t_{\mathcal{S}_1}(S) \\ \pi_{\mathcal{S}_1}^{\lfloor t/2\rfloor} S'=S''}}  \lVert \mathbf{1}_{S'} \{T\mathbf{1}_{S\setminus S''} - T\mathbf{1}_{S\setminus S''}(x_{S'})\}\rVert_{p_1}^{p_1}\,.
\end{align*}
Observe that the cubes $S'$  are contained in $S''$, and they are disjoint.
The Local Testing Condition implies the inequality $\alpha\le \mathbf{T}_{\textup{loc}}^{p_1}\mu(S'')$.
In order to analyse term $\beta$, we fix $S'\in\textup{ch}^t_{\mathcal{S}_1}(S)$
such that $\pi_{\mathcal{S}_1}^{\lfloor t/2\rfloor} S'=S''$. Since
$\lfloor t/2\rfloor \ge r+1$, we have
$2^r\ell S' < \ell S''$.
By construction of
the stopping cubes, either $S'$ or $\pi_1 S'$ is  good. In both of these
cases, by goodness\footnote{
This application is the principal motivation for our definition
of goodness; recall that 
good cubes are neither $1$-bad nor $2$-bad.
The same application arises also later,  Lemma
\ref{l.tss_2}.
},  we have $\ell S' \le \mathrm{dist}(S',\partial S'')$.
Hence, by the off-diagonal estimate
 \eqref{e.off_diagonal}, we have
$\lvert T\mathbf{1}_{S\setminus S''}(x) - T\mathbf{1}_{S\setminus S''}(x_{S'})\rvert 
\lesssim 1$ if $ x\in S'$.
This inequality allows us to conclude that 
$\beta\lesssim \mu(S'')$.
\end{proof}

\subsection{Summation involving cubes $R\subset S$}
Here we show the inequality,
\begin{equation}\label{e.subset_ineqII}
\bigg\lvert \sum_{S\in\mathcal{S}_1}  {\mathbf{B}}^{\textup{para}}_{S,\subset}(f_1,f_2)  \bigg\rvert
\lesssim 1+\mathbf{T}_{\textup{loc}}\,.
\end{equation}
Let us 
fix $S\in\mathcal{S}_1$, and
express the series defining $ {\mathbf{B}}^{\textup{para}}_{S,\subset}(f_1,f_2)$ in a convenient manner. For
a cube $S'\in{\mathcal{S}_1}$,
we denote by $\mathcal{R}(S')$ the family of maximal cubes
in  $\{R'\in\mathcal{S}_2\,:\, \pi_{\mathcal{S}_1}R'=S'\}$; this can be an empty family.
By defining constants $\varepsilon_{Q,S}$ as solutions to \eqref{epsilon_def}, we can write
${\mathbf{B}}^{\textup{para}}_{S,\subset}(f_1,f_2)$ as
\begin{equation}\label{e.maximal_identity}
\begin{split}
 \sigma_1(S)\sum_{t,k\ge 0}
\sum_{ \substack{S'\in\textup{ch}^t_{\mathcal{S}_1}(S)}}
\sum_{R\in\mathcal{R}(S')}
 \sum_{\substack{R'\in\textup{ch}^k_{\mathcal{S}_2}(R) \\ \pi_{\mathcal{S}_1} R'=S'}}
\sum_{Q\in\mathcal{G}_2:\pi_{\mathcal{S}_2}Q=R'} \langle 
\mathbf{1}_{R'}\{T\mathbf{1}_S - \tau_{t,k,S',R'}\},
\varepsilon_{Q,S} \Delta_Q f_2\rangle\,,
\end{split}
\end{equation}
where we have denoted
\[
\tau_{t,k,S',R'} = \begin{cases} 0\,,\quad &\text{ if } t,k\in \{0,\ldots, 2r+1\}\,;\\
T\mathbf{1}_{S\setminus \pi_{\mathcal{S}_2}^{\lfloor k/2\rfloor} R'}(x_{R'})\,, &
\text{ if }k\ge 2(r+1)\,;\\
T\mathbf{1}_{S\setminus \pi_{\mathcal{S}_1}^{\lfloor t/2\rfloor} S'}(x_{S'})\,, &\text{ otherwise}\,.
\end{cases}
\]
It will be convenient to denote for all $t\ge 0$,
\[
\mathbf{A}_{t,S} = \bigg\{
\sum_{k\ge 0}
\sum_{ \substack{S'\in\textup{ch}^t_{\mathcal{S}_1}(S)}}
\sum_{R\in\mathcal{R}(S')}
 \sum_{\substack{R'\in\textup{ch}^k_{\mathcal{S}_2}(R) \\ \pi_{\mathcal{S}_1} R'=S'}}\lVert \mathbf{1}_{R'}
  \{T\mathbf{1}_S - \tau_{t,k,S',R'}\}\rVert_{p_1}^{p_1}
\bigg\}^{1/p_1}\,.
\]
The useful inequality
$\sup_{\mathcal{G}_2\times\mathcal{S}_1}\lvert \varepsilon_{Q,S}\rvert\lesssim 1$ is a consequence of Lemma \ref{l.epsilon_bounded}.
Hence, by  Lemma \ref{l.tss_2} and H\"older's inequality,
combined with
inequality \eqref{e.mart_ineq},
\begin{align*}
\lvert{\mathbf{B}}^{\textup{para}}_{S,\subset}(f_1,f_2)\rvert 
&\lesssim  \sum_{t\ge 0} \sigma_1(S)\mathbf{A}_{t,S}\bigg\{
\sum_{ \substack{S'\in\textup{ch}^t_{\mathcal{S}_1}(S)}}
\sum_{\substack{R'\in\mathcal{S}_2 \\\pi_{\mathcal{S}_1} R'=S'}}
\bigg\lVert \sum_{Q\in\mathcal{G}_2:\pi_{\mathcal{S}_2}Q=R'}\varepsilon_{Q,S} \Delta_Q f_2 \bigg\rVert_{p_2}^{p_2}\bigg\}^{1/p_2}\\
&\lesssim (1+\mathbf{T}_{\textup{loc}})
\sum_{t\ge 0}2^{-t/p_1}  \sigma_1(S) \mu(S)^{1/p_1}\bigg\{
\sum_{ \substack{S'\in\textup{ch}^t_{\mathcal{S}_1}(S)}}
\sum_{\substack{R'\in\mathcal{S}_2 \\\pi_{\mathcal{S}_1} R'=S'}}\lVert P_{2,R'}  f_2\rVert_{p_2}^{p_2}\bigg\}^{1/p_2}\,.
\end{align*}
The very last upper bound is summable in $S\in \mathcal{S}_1$.
 Indeed, after
changing the order of $S$ and $t$ summations,
an application
of inequalities \eqref{e.quasi} and \eqref{e.proj_ineq} leaves us  a geometric series in $t$.
The proof of inequality \eqref{e.subset_ineqII} is complete.

\begin{lemma}\label{l.tss_2}
For each $S\in\mathcal{S}_1$ and $t\ge 0$, we have
$\mathbf{A}_{t,S}\lesssim (1+\mathbf{T}_{\textup{loc}})2^{-t/p_1}\mu(S)^{1/p_1}$.
\end{lemma}

\begin{proof}
Let us make a case study, and first assume that $t\in \{0,\ldots, 2r+1\}$.
We split  $\mathbf{A}_{t,S}^{p_1}$ in two subseries, subject
to $k\in \{0,\ldots,2r+1\}$ and $k\ge 2(r+1)$. For a fixed $k\in \{0,\ldots,2r+1\}$, we
rely on disjointness properties of layers and maximal cubes in order to see that
\[
\sum_{ \substack{S'\in\textup{ch}^t_{\mathcal{S}_1}(S)}}
\sum_{R\in\mathcal{R}(S')}
 \sum_{\substack{R'\in\textup{ch}^k_{\mathcal{S}_2}(R) \\ \pi_{\mathcal{S}_1} R'=S'}}\lVert \mathbf{1}_{R'} \{T\mathbf{1}_S - \tau_{t,k,S',R'}\}\rVert_{p_1}^{p_1} \le \lVert \mathbf{1}_S T\mathbf{1}_{S}\rVert_{p_1}^{p_1}
 \le \mathbf{T}_{\textup{loc}}^{p_1}\mu(S)\,.
\]
Applying these inequalities with finite number of indices $k\in \{0,\ldots, 2r+1\}$ shows the required inequality for the first subseries.
The second subseries is bounded by
\begin{equation}\label{e.second}
\begin{split}
\sum_{k\ge 2r+2}&
\sum_{ \substack{S'\in\textup{ch}^t_{\mathcal{S}_1}(S)}}
\sum_{R\in\mathcal{R}(S')}
\sum_{R''\in\textup{ch}^{\lceil k/2\rceil}_{\mathcal{S}_2}(R)}
 \sum_{\substack{R'\in\textup{ch}^k_{\mathcal{S}_2}(R) \\ 
 \pi_{\mathcal{S}_2}^{\lfloor k/2\rfloor} R'=R''}}\lVert \mathbf{1}_{R'}\{T\mathbf{1}_S - T\mathbf{1}_{S\setminus R''}(x_{R'})\}\rVert_{p_1}^{p_1}\\
 &\lesssim (1+\mathbf{T}_{\textup{loc}})^{p_1} \sum_{k\ge 2r+2} 2^{-k}
\sum_{ \substack{S'\in\textup{ch}^t_{\mathcal{S}_1}(S)}}
\mu(S')
\lesssim (1+\mathbf{T}_{\textup{loc}})^{p_1}\mu(S)\,.
\end{split}
\end{equation}
In the first step above, we applied a simple modification of inequality \eqref{e.auxiliary_est} and 
sparsness property of $\mathcal{S}_2$, we refer to Remark \ref{r.sparse}.

Let us then focus on the case of $t\ge 2(r+1)$. Again, we split the
series $\mathbf{A}_{t,S}^{p_1}$ in two subseries as before.
For the first subseries, associated with indices $k\in \{0,\ldots,2r+1\}$, we
use inequality
\begin{align*}
\sum_{ \substack{S'\in\textup{ch}^t_{\mathcal{S}_1}(S)}}
\sum_{R\in\mathcal{R}(S')} &
 \sum_{\substack{R'\in\textup{ch}^k_{\mathcal{S}_2}(R) \\ \pi_{\mathcal{S}_1} R'=S'}}  \lVert \mathbf{1}_{R'} \{T\mathbf{1}_S - \tau_{t,k,S',R'}\}\rVert_{p_1}^{p_1}
\\& \le 
\sum_{\substack{ S''\in\textup{ch}^{\lceil t/2\rceil}_{\mathcal{S}_1}(S)} }
\sum_{\substack {S'\in\textup{ch}^t_{\mathcal{S}_1}(S) \\ \pi_{\mathcal{S}_1}^{\lfloor t/2\rfloor} S'=S''}} 
\lVert \mathbf{1}_{S'} \{T\mathbf{1}_S-T\mathbf{1}_{S\setminus S''}(x_{S'})\}\rVert_{p_1}^{p_1}\,,
\end{align*}
and then proceed as in the proof of Lemma \ref{l.tss'_1}.
Finally, the second subseries is bounded by LHS\eqref{e.second} which, in
turn, is controlled by  $\lesssim (1+\mathbf{T}_{\textup{loc}})^{p_1}2^{-t}\mu(S)$, Remark \ref{r.sparse}.
\end{proof}

\section{The Inside-Stopping/Error Term}\label{s.remaining}

In the present section, we concentrate on the two terms, labelled as
`stopping'  and `error', that
were introduced in the beginning of Section
\ref{s.inside_core}.  
We aim to prove the following proposition.

\begin{proposition}\label{p.inside_core}
We have
$\big\lvert {\mathbf{B}}^{\textup{stop}}(f_1,f_2)  \big\rvert
+\big\lvert {\mathbf{B}}^{\textup{error}}(f_1,f_2) \big\rvert 
\lesssim 1$.
\end{proposition}


\subsection{The stopping term}
The stopping term ${\mathbf{B}}^{\textup{stop}}(f_1,f_2)$ is written as
$\sum_{t=r+1}^\infty {\mathbf{B}}^{\textup{stop}}_t(f_1,f_2)$,
\begin{align*}
\lvert {\mathbf{B}}^{\textup{stop}}_t(f_1,f_2)\rvert &= 
\bigg\lvert
\sum_{P\in\mathcal{G}_1} \sum_{\substack{Q\in\mathcal{G}_2 \\ 2^t \ell Q = \ell P}} 
\mathbf{1}_{Q\subset P}
\langle{\Delta}_P f_1 \rangle_{P_Q}
\langle T \mathbf{1}_{\pi_{\mathcal{S}_1} P_Q\setminus P_Q}, \Delta_Q f_2\rangle\bigg\rvert\\
&\lesssim 
2^{-t\eta(1-\gamma)}\int_{\mathbf{R}^n} \sum_{P\in\mathcal{G}_1} \sum_{\substack{
Q\in\mathcal{G}_2 \\ 2^t \ell Q = \ell P}} 
\mathbf{1}_Q(x)
\lvert {\Delta}_P f_1 (x)\rvert\cdot
\mathbf{1}_{Q\subset P} \lvert \Delta_Q f_2(x)\rvert \,d\mu(x)\,.
\end{align*}
In the last step, we used the off-diagonal estimate \eqref{e.off_diagonal}
and the fact that $\Delta_Q f_2$ has mean zero.
Applying Cauchy--Schwarz and  H\"older's inequality,
and then observing inequalities,
\begin{equation}\label{e.crucial}
\sum_{\substack{P\in\mathcal{G}_1 \\ 2^t\ell Q = \ell P}} \mathbf{1}_{Q\subset P}
\le {1} \quad (Q\in\mathcal{G}_2)\,,\qquad \sum_{\substack{Q\in\mathcal{G}_2 \\ 2^t\ell Q = \ell P}} 
\mathbf{1}_Q \le \mathbf{1}_{\mathbf{R}^n} \quad (P\in\mathcal{G}_1)\,,
\end{equation}
we obtain, for a fixed $t\ge r+1$,
\begin{align*}
\lvert {\mathbf{B}}^{\textup{stop}}_t(f_1,f_2)\rvert &\lesssim 
2^{-t\eta(1-\gamma)} 
\bigg\lVert \bigg(
\sum_{P\in\mathcal{G}_1} \lvert {\Delta}_P f_1\lvert^2\bigg)^{1/2}\bigg\rVert_{p_1}\
\bigg\lVert \bigg(
\sum_{Q\in\mathcal{G}_2} \lvert \Delta_Q f_2\lvert^2\bigg)^{1/2}\bigg\rVert_{p_2}
\lesssim 2^{-t\eta(1-\gamma)}\,.
\end{align*}
In the penultimate step, we used inequality \eqref{kahane}.
The last bound is summable in $t$, and this
concludes analysis of the stopping term.

%

\subsection{The error term}
We write ${\mathbf{B}}^{\textup{error}}(f_1,f_2) =\sum_{t=r+1}^\infty {\mathbf{B}}^{\textup{error}}_t(f_1,f_2)$,
\[
{\mathbf{B}}^{\textup{error}}_t(f_1,f_2)= \sum_{j=1}^{2^n} \sum_{Q\in\mathcal{G}_2} \sum_{\substack{P\in\mathcal{G}_1 \\  2^t\ell Q = \ell P }}  \mathbf{1}_{Q\subset P}
\mathbf{1}_{P_j\not=P_Q} \langle {\Delta}_P f_1\rangle_{P_j} \langle T \mathbf{1}_{P_j}, \Delta_Q f_2\rangle\,.
\]
Let us denote $T_{P_j,Q} = 
\mathbf{1}_{P_j\not = P_Q} \{T\mathbf{1}_{P_j}
- T\mathbf{1}_{P_j}(x_Q)\}$.
By the fact that $\Delta_Q f_2$ has mean zero, we can bound  $\lvert {\mathbf{B}}^{\textup{error}}_t(f_1,f_2)\rvert$ with $t\ge r+1$ by
\begin{align*}
& \sum_{j=1}^{2^n} \bigg\lvert 
 \int_{\mathbf{R}^n}
 \sum_{Q\in\mathcal{G}_2} 
 \Delta_Q f_2(x) \cdot \mathbf{1}_Q(x)
 \sum_{\substack{P\in\mathcal{G}_1 \\  2^t\ell Q = \ell P }}  \mathbf{1}_{Q\subset P}
\mathbf{1}_{P_j\not=P_Q} \langle {\Delta}_P f_1\rangle_{P_j}   T_{P_j,Q}(x) \,d\mu(x)
\bigg\rvert\\
&\le \mathbf{A}_{t}\cdot \bigg\lVert \bigg(\sum_{Q\in\mathcal{G}_2} \lvert \Delta_Q f_2\rvert^2\bigg)^{1/2}\bigg\rVert_{p_2}\lesssim
\mathbf{A}_t\,,
\end{align*}
where we have denoted 
\begin{align*}
&\mathbf{A}_{t} = 
\sum_{j=1}^{2^n} \bigg\lVert  
\bigg(
\sum_{k\in\mathbf{Z}} \bigg\lvert 
\sum_{\substack{P\in\mathcal{G}_{1,k+t}}}
\langle {\Delta}_P f_1\rangle_{P_j}
\sum_{\substack{ Q\in\mathcal{G}_{2,k}}}
 \mathbf{1}_{Q\subset P}
\mathbf{1}_{P_j\not=P_Q}    \mathbf{1}_Q T_{P_j,Q} \bigg\rvert^2\bigg)^{1/2}
 \bigg\rVert_{p_1}\,.
\end{align*}
By an off-diagonal estimate for $T_{P_j,Q}$, i.e. Lemma \ref{l.s_estimate}
applied to cubes $P_j$ and $Q$,
\begin{equation}\label{e.transition_est}
\bigg\lvert \sum_{\substack{ Q\in\mathcal{G}_{2,k}}}
  \mathbf{1}_{Q\subset P}
\mathbf{1}_{P_j\not=P_Q}    \mathbf{1}_Q(x) T_{P_j,Q}(x)\bigg\rvert \lesssim
2^{-t \eta/4} \mathbf{1}_P(x) \mu(P_j) \mu(P)^{-1},\qquad x\in \mathbf{R}^n\,.
\end{equation}
Thus, by inequalities \eqref{stein} and  \eqref{kahane},
\begin{align*}
\mathbf{A}_{t} 
&\lesssim 2^{-t\eta/4}\bigg\lVert  
\bigg(
\sum_{k\in\mathbf{Z}} \bigg\lvert 
\sum_{\substack{P\in\mathcal{G}_{1,k+t}  }}
\langle \lvert {\Delta}_P  f_1\rvert \rangle_{P}\mathbf{1}_P
\bigg\rvert^2\bigg)^{1/2}
\bigg\rVert_{p_1}\\
& \le 2^{-t\eta/4}\bigg\lVert  
\bigg(
\sum_{k\in\mathbf{Z}} \big(
\mathbf{E}_{1,k+t} \lvert {\Delta}_{k+t} f_1\rvert\big)^2\bigg)^{1/2}\bigg\lVert_{p_1}\lesssim 2^{-t\eta/4}\lVert  f_1\rVert_{p_1}\lesssim 2^{-t\eta/4}\,.
\end{align*}
The last
bound is summable in $t\ge r+1$. The
 proof of Proposition \ref{p.inside_core} is complete.

\section{The Separated Term}\label{s.separated}

Here we treat the separated term, we refer
to Section \ref{s.decomposition}.

\begin{proposition}\label{p.separated} We have inequality
$\lvert \mathbf{B}_{\textup{separated}}(f_1,f_2)\rvert \lesssim 1$.
\end{proposition}

For the proof, we need
preparations.
Recall that $D(Q,P)=\ell Q + \mathrm{dist}(Q,P)+\ell P$, and
 write
$D(Q,P)/\ell P \sim 2^u$ if
$2^{u} < {D(Q,P)}/\ell P \le 2^{u+1}$. 
The separated term is a sum over $u,m\in \mathrm{N}_0$ 
and $j\in \{1,2,\ldots, 2^n\}$ of terms
\[
\mathbf{B}^{u,m,j}(f_1,f_2) =
\sum_{k\in\mathbf{Z}} \sum_{Q\in\mathcal{G}_{2,k-m}} 
\sum_{ \substack{P\in\mathcal{G}_{1,k} \\ D(Q,P)/\ell P \sim 2^u}}
\mathbf{1}_{\ell Q \le \mathrm{dist}(Q,P)} \langle \Delta_P f_1\rangle_{P_j}
\langle T \mathbf{1}_{P_j}, \Delta_Q f_2\rangle\,.
\]
For $Q$ and $P_j$ as in the summation above, let us write
$T_{P_j,Q} = 
\mathbf{1}_{\ell Q \le \mathrm{dist}(Q,P)} \{T\mathbf{1}_{P_j}
- T\mathbf{1}_{P_j}(x_Q)\}$.
Since $\Delta_Q f_2$ has mean zero, we can write 
$\lvert \mathbf{B}^{u,m,j}(f_1,f_2)\rvert$ as
\begin{align*}
&\bigg\lvert \int_{\mathbf{R}^n}\sum_{k\in\mathbf{Z}} \sum_{Q\in\mathcal{G}_{2,k-m}} 
\Delta_Q f_2(x)\cdot \mathbf{1}_Q(x)\sum_{ \substack{P\in\mathcal{G}_{1,k} \\ D(Q,P)/\ell P \sim 2^u}}
\langle \Delta_P f_1\rangle_{P_j} T_{P_j,Q}(x)\,d\mu(x)\bigg\rvert\\
&\le
 \mathbf{A}_{u,m,j}\cdot \bigg\lVert \bigg( \sum_{k\in\mathbf{Z}} \sum_{Q\in\mathcal{G}_{2,k-m}} 
 \lvert \Delta_Q f_2\rvert^2 \bigg)^{1/2}\bigg\rVert_{p_2}
 \lesssim  \mathbf{A}_{u,m,j}\,,
\end{align*}
where we have denoted
\begin{align*}
&\mathbf{A}_{u,m,j} = \bigg\lVert
\bigg(\sum_{k\in\mathbf{Z}} \sum_{Q\in\mathcal{G}_{2,k-m}} 
\bigg\lvert \mathbf{1}_Q \sum_{ \substack{P\in\mathcal{G}_{1,k} \\ D(Q,P)/\ell P \sim 2^u}}
\langle \Delta_P f_1\rangle_{P_j} T_{P_j,Q}\bigg\rvert^2\bigg)^{1/2}\bigg\rVert_{p_1}\,.
\end{align*}
In order to finish the proof of Proposition \ref{p.separated},
we invoke the following lemma.

\begin{lemma}\label{l.useful}
For $u,m\in\mathbf{N}_0$ and $j\in \{1,2,\ldots, 2^n\}$, we have
$\mathbf{A}_{u,m,j} \lesssim 2^{-\eta (m+u)/4}$.
\end{lemma}

\begin{proof}
For each $S\in\mathcal{D}_{1,k+u+\theta(u+m)}$, $k\in\mathbf{Z}$,  we consider
the kernel
\begin{equation}\label{kernel_S}
K_S(x,y) =  \sum_{\substack{ P\in\mathcal{G}_{1,k} \\ P\subset S} } 
 \sum_{ \substack{Q\in\mathcal{G}_{2,k-m} \\ D(Q,P)/\ell P \sim 2^u}}
 \mathbf{1}_Q(x) \cdot \widetilde{T}_{P_j,Q}(x)\cdot \mathbf{1}_{P_j}(y)\,,
\end{equation}
where $\widetilde{T}_{P_j,Q}=\mathbf{1}_{\ell Q\le \mathrm{dist}(Q,P)}\widetilde{T}_{P_j,Q}$ is defined by
\[
\frac{T_{P_j,Q}(x)}{\mu(P_j)} = 2^{-\eta(m+u)/4} \cdot \frac{\widetilde{T}_{P_j,Q}(x)}{\mu(S)}\,.
\]
By   Lemma \ref{l.s_exists_s} and Lemma \ref{l.s_estimate}, 
\begin{equation}\label{kernel_est_sep}
\begin{split}
\lvert K_S(x,y)\rvert &\lesssim  \sum_{
\substack{P\in\mathcal{G}_{1,k} \\ P\subset S}} 
 \sum_{ \substack{Q\in\mathcal{G}_{2,k-m} \\ D(Q,P)/\ell P \sim 2^u}}
 \mathbf{1}_Q(x) \cdot \mathbf{1}_{P_j}(y)
 \le\mathbf{1}_S(x) \cdot \mathbf{1}_S(y)\,.
\end{split}
\end{equation}
We can now finish the proof as follows.  Inequality \eqref{kernel_est_sep}
allows us to write $2^{\eta p_1(m+u)/4}\mathbf{A}_{u,m,j}^{p_1}$ as
\begin{align*}
&
\int_{\mathbf{R}^n} \bigg( \sum_{k\in\mathbf{Z}} \bigg\lvert
\sum_{S\in \mathcal{D}_{1,k+u+\theta(u+m)}} \frac{1}{\mu(S)} \int_{\mathbf{R}^n} 
K_S(x,y) \Delta_k f_1(y)\,d\mu(y)\bigg\rvert^2 \bigg)^{p_1/2}\,d\mu(x)\\
&\lesssim \int_{\mathbf{R}^n} \bigg( \sum_{k\in\mathbf{Z}} \bigg\lvert
\underbrace{\sum_{S\in \mathcal{D}_{1,k+u+\theta(u+m)}} \langle \lvert \Delta_k f_1\rvert \rangle_S \mathbf{1}_S(x)}_{=\mathbf{E}_{1,k+u+\theta(u+m)} \lvert \Delta_k f_1\rvert(x)}\bigg\rvert^2 \bigg)^{p_1/2}\,d\mu(x)\,.
\end{align*}
Appealing to inequalities \eqref{stein} and \eqref{kahane} shows
that $\mathbf{A}_{u,m,j}\lesssim 2^{-\eta(m+u)/4}$.
\end{proof}

\section{Preparations for the Nearby Term}\label{s.nearby1}

The surgery argument for the nearby term follows
\cite{1201.0648} but there are also essential differences.
Let us abbreviate  $(P,Q)\in\mathcal{P}_{\textup{nearby}}$ as $P\sim Q$.
Hence, the conditions for $P\sim Q$
are
\begin{equation}\label{rems}
(P,Q)\in\mathcal{G}_1\times\mathcal{G}_2\,,\qquad 2^{-r}\ell P\le \ell Q\le \ell P\,,\quad \mathrm{dist}(Q,P)<\ell Q=\ell Q\wedge \ell P.
\end{equation} 
In particular $\ell Q\le \ell P\le D(Q,P)\le (2^r+2)\ell Q$, i.e., these quantities are comparable if $P\sim Q$.
During the course of the remaining sections, we will prove the following 
proposition.
\begin{proposition}
For a fixed $t>p_1\vee p_2$, we have
\begin{align*}
\mathbf{E}_{\omega_1}\mathbf{E}_{\omega_2}\lvert \mathbf{B}_{\textup{nearby}}(f_1,f_2)\rvert 
&\le
C(r,\upsilon,\epsilon)(1+\mathbf{T}_{\textup{loc}})
+\big(C(r,\upsilon)\epsilon^{1/t} + C(r)\upsilon^{1/t}\big)\mathbf{T}\,.
\end{align*}
Aside from the indicated absorption parameters, the constants on the right hand
side can depend upon the parameters $n,p_1,p_2,\eta,\mu$.
\end{proposition}

\begin{remark}
We shall
track dependence of various inequalities
on absorption parameters: $r,\upsilon,\epsilon$. 
There is no need to do this quantitatively,
and thus we agree upon the following convenient notation: 
$C(r)$, $C(r,\upsilon)$, and $C(r,\upsilon,\epsilon)$ denote positive
numbers that
are allowed to depend on the indicated
absorption parameters, but also on
parameters $n,p_1,p_2,\eta,\mu$. Moreover, the 
value of these numbers is allowed to vary from one occurrence to another.
\end{remark}

For a given $P\in\mathcal{G}_1$ there are at most $C(r)$
cubes $Q\in\mathcal{G}_2$ satisfying \eqref{rems}. Hence,
without loss of generality,
it suffices consider a finite number of subseries of the general form
\begin{equation}\label{subser}
\mathbf{E}_{\omega_1}\mathbf{E}_{\omega_2}\bigg|\sum_{P\in\mathcal{G}_1}  \langle 
T\Delta_P f_1, \Delta_Q f_2\rangle\bigg|\,,
\end{equation}
where $Q=Q(P)\in\mathcal{G}_2\cup\{\emptyset\}$ inside the summation satisfies $P\sim Q$ or $Q=\emptyset$.\footnote{We agree that $\Delta_\emptyset f_2=0$.}
At the same time, we can also assume
that for any $Q\in\mathcal{G}_2$ there is at most
one $P\in\mathcal{G}_1$ such that $Q=P(Q)$.
We fix one series like this, and  the convention that $Q$ is implicitly a function
of $P$ will be maintained. 

\subsection{First reductions}
We immediately find that
\eqref{subser}
is dominated by
\begin{equation}\label{eeqa}
\sum_{i,j=1}^{2^n} \mathbf{E}_{\omega_1}\mathbf{E}_{\omega_2}\bigg|\sum_{P\in\mathcal{G}_1} 
\langle \Delta_P {f}_1\rangle_{P_j}
\langle T\mathbf{1}_{P_j}, \mathbf{1}_{Q_i}\rangle\langle \Delta_Q {f}_2 \rangle_{Q_i}\bigg|\,.
\end{equation}
Fix $i,j\in \{1,\ldots,n\}$.
For a cube $R$ in $\mathbf{R}^n$, define an `$\upsilon$-boundary region':
$\delta_R^\upsilon= (1+\upsilon)R\setminus (1-\upsilon)R$.
If $P\in\mathcal{D}_1$ and $Q=Q(P)\not=\emptyset$, we write
\begin{equation}\label{e.basic_one}
\begin{split}
&Q_{i,\partial}=Q_i\cap \delta^\upsilon_{P_j};\quad Q_{i,\mathrm{sep}} =(Q_i\setminus Q_{i,\partial})\setminus (Q_i\cap P_j);\quad \Delta_{Q_i}=(Q_i\cap P_j)\setminus (Q_{i,\partial});\\
&P_{j,\partial}=P_j\cap \delta^\upsilon_{Q_i};\quad P_{j,\mathrm{sep}} =(P_j\setminus P_{j,\partial})\setminus (Q_i\cap P_j);\quad \Delta_{P_j}=(Q_i\cap P_j)\setminus P_{j,\partial}.
\end{split}
\end{equation}
For an illustration of these sets, we refer to Figure 1.
\begin{figure}[!htb]
\begin{center}
\includegraphics[scale=0.4,viewport=20 120 580 720,clip,angle=0]{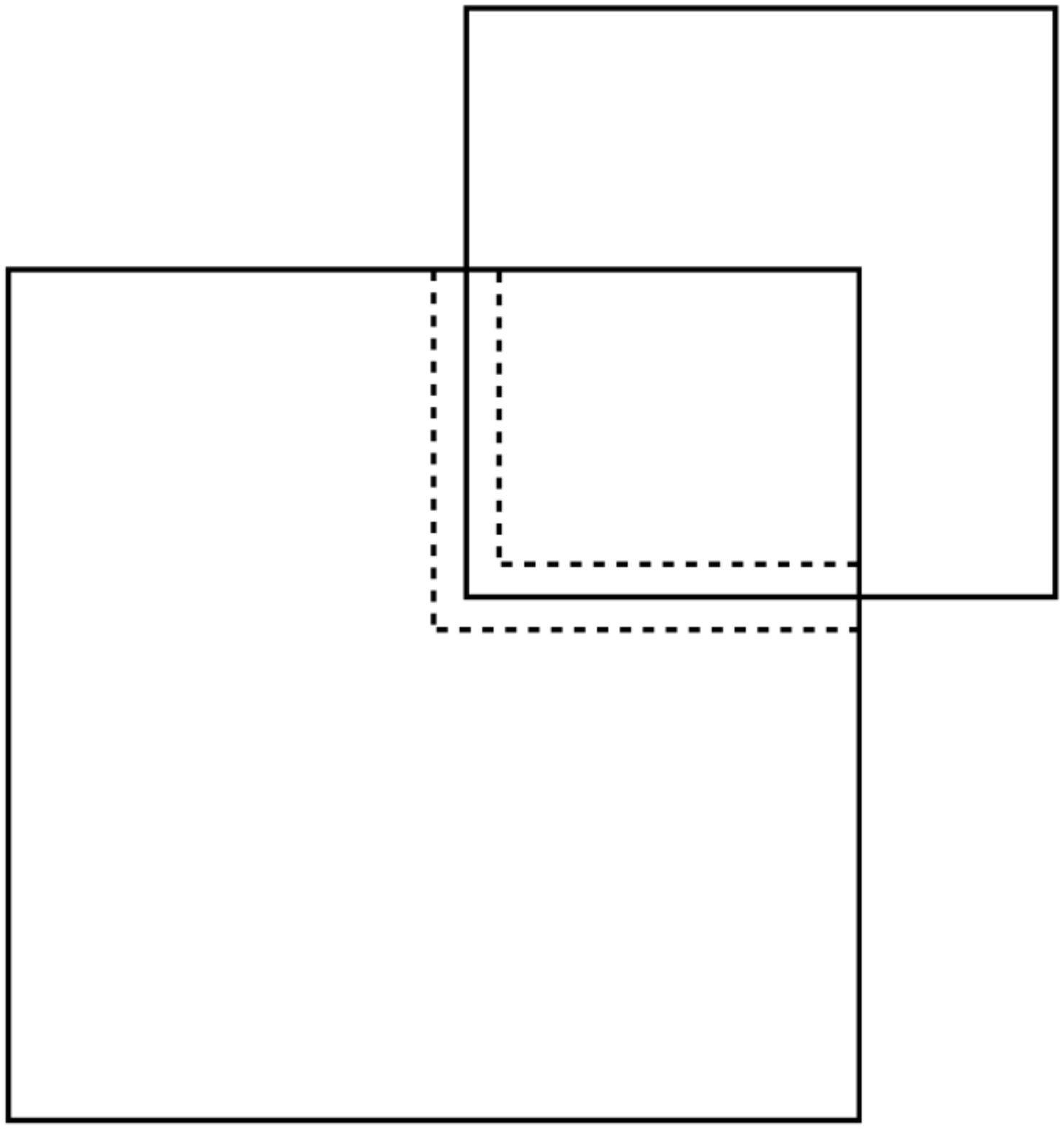}
\end{center}
\caption{
The larger cube is $P_j$, and the smaller cube is $Q_i$.
The dashed line segments separate sets $P_{j,sep}$, $P_{j,\partial}$,
and $\Delta_{P_j}$ from each other. } 
\end{figure}

We
write the matrix coefficient $\langle T\mathbf{1}_{P_j},\mathbf{1}_{Q_i}\rangle$ in \eqref{eeqa}  as
\begin{equation}\label{tenf}
\begin{split}
& \langle T\mathbf{1}_{P_{j,\mathrm{sep}}}
,\mathbf{1}_{Q_i}\rangle + \langle T\mathbf{1}_{P_{j,\partial}}, \mathbf{1}_{Q_i}\rangle
+\langle T\mathbf{1}_{\Delta_{P_j}}, \mathbf{1}_{\Delta_{Q_i}}\rangle
 +\langle T\mathbf{1}_{\Delta_{P_j}}, \mathbf{1}_{Q_{i,\partial}}\rangle + \langle
T\mathbf{1}_{\Delta_{P_j}} ,\mathbf{1}_{Q_{i,\mathrm{sep}}}\rangle\,,
\end{split}
\end{equation}
and these are denoted by
$M_1(P)+M_2(P)+M_3(P)+M_4(P)+M_5(P)$, respectively.

\subsection{Description of different terms}
The
heart of the argument lies in estimating  terms
\[
M_3(P)=\langle \mathbf{1}_{\Delta_{P_j}},T\mathbf{1}_{\Delta_{Q_i}}\rangle=\alpha_1(P)+\alpha_2(P)+\alpha_3(P),
\]
where the last decomposition depends
on a third random dyadic system $\mathcal{D}_{3}$, we refer to \eqref{ekahaj}. Terms $\alpha_2(P)$
and $\alpha_3(P)$, along with $M_2(P)$ and $M_4(P)$, are 
`$\upsilon$-boundary' terms.
The `separated'  terms $M_1(P)$ and $M_5(P)$ are treated by kernel size condition.

The term $\alpha_1(P)$ will further
be expanded in \eqref{alpha1} as 
\[\alpha_1(P)=\beta_1(P)+\beta_2(P)+\beta_3(P),\] where
$\beta_1(P)$ and $\beta_2(P)$ are so called
`$\epsilon$-boundary' terms. 
The local testing conditions and kernel size estimates are exploited in 
estimating `intersecting' term $\beta_3(P)$.

\subsection{Decomposition of $M_3(P)$}
Without loss of generality, we can assume that $\Delta_{Q_i}\not=\emptyset$ and $\Delta_{P_j}\not=\emptyset$. Indeed, otherwise we already have $M_3(P)=0$.

We introduce a third random dyadic system 
$\mathcal{D}_{3}=\mathcal{D}(\omega_3)$ 
that is independent of both $\mathcal{D}_1$ and $\mathcal{D}_2$.
Fix $j(\upsilon)\in\mathrm{Z}$ such that
$\upsilon/64\le 2^{j(\upsilon)}<\upsilon/32$.
Then, for every $P\in\mathcal{G}_1$ with
$Q=Q(P)\not=\emptyset$, we define a layer
\[\mathcal{L}=\mathcal{L}(P,\upsilon):=\mathcal{D}_{3,\log_2(s)}\] 
of $\mathcal{D}_3$-cubes with side length \begin{equation}\label{sl}
s=2^{j(\upsilon)}\ell Q_i=2^{j(\upsilon)}\cdot (\ell Q_i\wedge \ell P_j)\,.\end{equation}
That is, $\mathcal{L}$ is a layer of $\mathcal{D}_{3}$ that depends on 
parameters $P$  and $\upsilon$.

\begin{figure}[!htb]
\begin{center}
\includegraphics[scale=0.35,viewport=30 100 570 750,clip,angle=90]{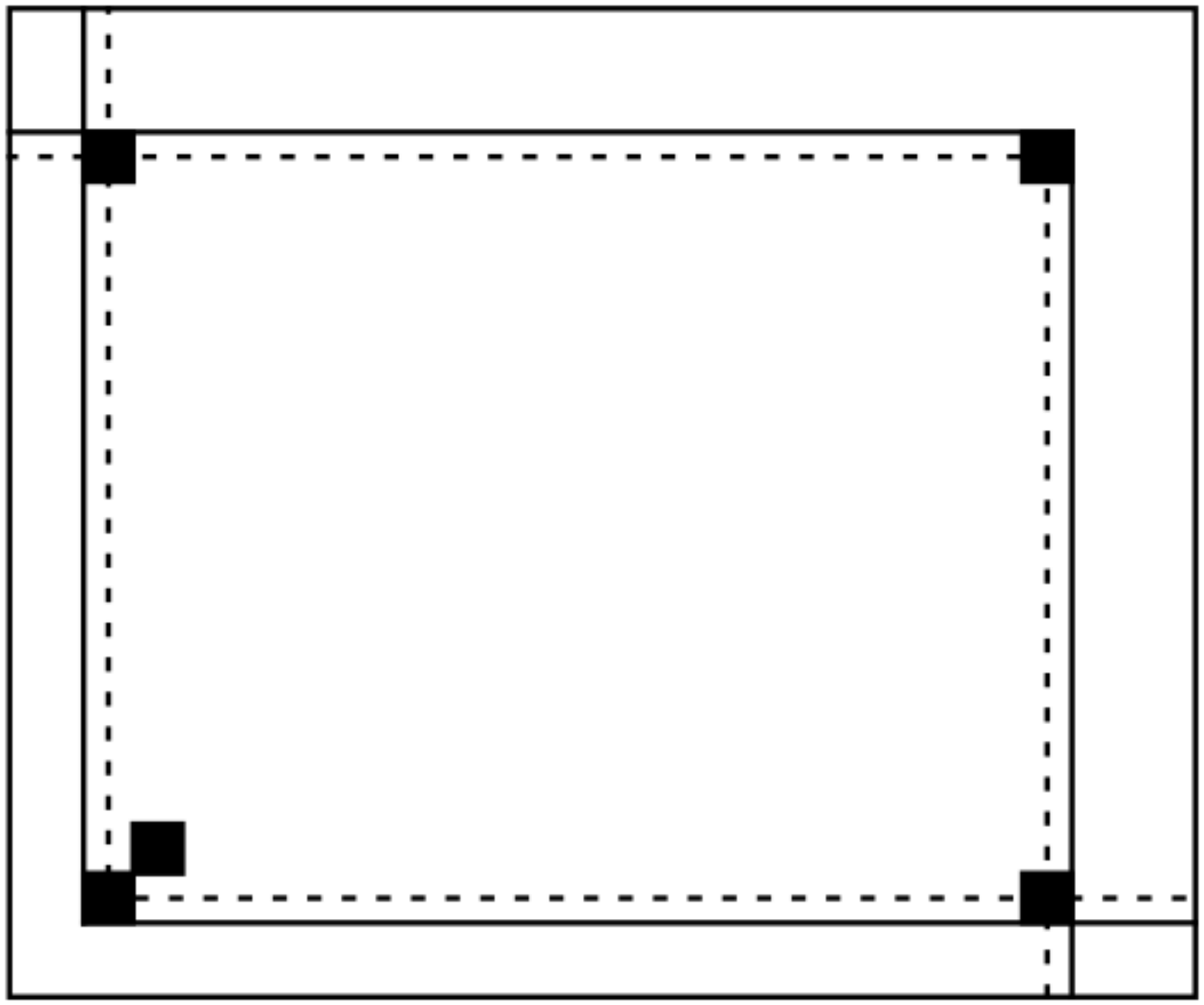}
\end{center}
\caption{
\small{Parallelogram is $Q_i\cap P_j$. 
Interiors of some black $\mathcal{L}$-cubes intersect
dashed line segments, which belong to the boundaries
of either  $\Delta_{P_j}$ or $\Delta_{Q_i}$. 
The $\mathcal{L}$-adjusted sets
$\Delta_{Q_i}^{\mathcal{L}}$ and $\Delta_{P_j}^{\mathcal{L}}$,
with solid boundaries, do
not intersect the indicated $\mathcal{L}$-cubes.}}
\end{figure}


Let $\Delta_{Q_i}^\mathcal{L},\Delta_{P_j}^\mathcal{L}\subset Q_i\cap P_j$ be
the following adaptations of $\Delta_{Q_i}$ and $\Delta_{P_j}$ to $\mathcal{L}$.
If necessary, we enlargen the latter sets
so that, for every $G\in\mathcal{L}$, either $G\cap \Delta_{Q_i}^\mathcal{L}=G\cap \Delta_{P_j}^\mathcal{L}=G$ 
or one of the two intersections
$G\cap \Delta_{Q_i}^\mathcal{L}$ and $G\cap \Delta_{P_j}^\mathcal{L}$ is empty. This is done
in such a way that we can write
\[
\Delta_{Q_i}^\mathcal{L}=\Delta_{Q_i}\cup \Delta^\partial_{Q_i}\,,\quad \Delta_{P_j}^\mathcal{L} = \Delta_{P_j}\cup \Delta^\partial_{P_j}\,,
 \]
both as disjoint unions, such that
$\Delta^\partial_{Q_i}\subset  Q_{i,\partial}\cap P_j$ and $\Delta^\partial_{P_j}\subset P_{j,\partial}\cap Q_i$. For an illustration,
we refer to Figure 2.

Now observe that $M_3(P) =\langle T\mathbf{1}_{\Delta_{P_j}},\mathbf{1}_{\Delta_{Q_i}} \rangle$ can be written as
\begin{equation}\label{ekahaj}
\begin{split}
\alpha_1(P)+\alpha_2(P)+\alpha_3(P)=
\langle T\mathbf{1}_{\Delta_{P_j}^\mathcal{L}},\mathbf{1}_{\Delta_{Q_i}^\mathcal{L}} \rangle
-\langle T\mathbf{1}_{\Delta_{P_j}^\partial}, \mathbf{1}_{\Delta_{Q_i}^\mathcal{L}} \rangle
-\langle T\mathbf{1}_{\Delta_{P_j}},\mathbf{1}_{\Delta_{Q_i}^\partial} \rangle\,.
\end{split}
\end{equation}
We remark that the terms in this decomposition depends on $\mathcal{D}_{3}$.

In order to define $\epsilon$-boundary terms, we let $P\in\mathcal{G}_1$ and write
\[
L_\epsilon=L_\epsilon(P,\upsilon) = \bigcup_{G\in \mathcal{L}(P,\upsilon)} \delta_G^\epsilon,\quad \delta_G^\epsilon=(1+\epsilon)G\setminus (1-\epsilon)G.
\]
We also write $\widetilde G=G\setminus L_\epsilon$ if $G\in\mathcal{L}=\mathcal{L}(P,\upsilon)$. Define
\[\Delta'_{Q_i}=\Delta^\mathcal{L}_{Q_i} \cap L_\epsilon\,,\quad \widetilde \Delta_{Q_i}=\Delta^\mathcal{L}_{Q_i}\setminus L_\epsilon\,,
\quad \Delta'_{P_j}=\Delta^\mathcal{L}_{P_j} \cap L_\epsilon\,,\quad \widetilde \Delta_{P_j}=\Delta^\mathcal{L}_{P_j}\setminus L_\epsilon\,.
\] 
Finally, we write $\alpha_1(P)=\langle T\mathbf{1}_{\Delta_{P_j}^\mathcal{L}} , \mathbf{1}_{\Delta_{Q_i}^\mathcal{L}} \rangle$ as
\begin{equation}\label{alpha1}
\begin{split}
\beta_1(P)+\beta_2(P)+\beta_3(P)= \langle T\mathbf{1}_{\Delta'_{P_j}},\mathbf{1}_{\Delta_{Q_i}^\mathcal{L}}\rangle+ \langle T\mathbf{1}_{\widetilde \Delta_{P_j}},\mathbf{1}_{\Delta'_{Q_i}}\rangle
+\langle T\mathbf{1}_{\widetilde \Delta_{P_j}},\mathbf{1}_{\widetilde \Delta_{Q_i}}\rangle\,.
\end{split}
\end{equation}


\section{The Nearby-Non-Boundary Term}\label{s.nearby2}

We estimate summations involving 
the 
separated terms $M_1(P)$ and $M_5(P)$, and
the intersecting term $\beta_3(P)$.
All of the estimates will be uniform over all three dyadic grids.

\subsection{Separated term}
The two indicators appearing in either $M_1(P)$ or $M_5(P)$
are associated with sets separated from each other.  
This observation will allow us to prove inequality
\begin{equation}\label{l.m1lem}
\bigg|
\sum_{\substack{P\in\mathcal{G}_1}} \langle \Delta_P f_1 \rangle_{P_j} (M_1(P)+M_5(P)) \langle 
\Delta_Q f_2 \rangle_{Q_i}
\bigg|
\le C(r,\upsilon)\,.
\end{equation}

\begin{proof}[Proof of inequality \eqref{l.m1lem}]
We focus on summation over the terms $M_1(P)$, and the treatment of summation over terms $M_5(P)$ is analogous.
We write
$T_{P_j,Q_i}=\mathbf{1}_{Q=Q(P)}\langle T\mathbf{1}_{P_{j,\mathrm{sep}}}
,\mathbf{1}_{Q_i}\rangle$.
Then, by inequalities \eqref{rems}, the term under focus can be written as
\begin{align*}
&\bigg\lvert \sum_{m=0}^r \sum_{u\in \{0,1\}}  \sum_{k\in\mathbf{Z}}
 \sum_{\substack{Q\in\mathcal{G}_{2,k-m}}}\sum_{\substack{P\in\mathcal{G}_{1,k}\\D(Q,P)/\ell P \sim 2^u }} \langle \Delta_P f_1\rangle_{P_j} T_{P_j,Q_i} \langle \Delta_Q f_2\rangle_{Q_i}\bigg\rvert\\
&\le  \sum_{m,u} \mathbf{A}_{m,u,i,j}\cdot \bigg\lVert \bigg( \sum_{k\in\mathbf{Z}} \sum_{Q\in\mathcal{G}_{2,k-m}} 
 \lvert \Delta_Q f_2\rvert^2 \bigg)^{1/2}\bigg\rVert_{p_2}
 \lesssim \sum_{m,u} \mathbf{A}_{m,u,i,j}\,,
\end{align*}
where we have denoted
\begin{align*}
&\mathbf{A}_{m,u,i,j} = \bigg\lVert
\bigg(\sum_{k\in\mathbf{Z}} \sum_{Q\in\mathcal{G}_{2,k-m}} 
\bigg\lvert \mathbf{1}_{Q_i} \sum_{ \substack{P\in\mathcal{G}_{1,k} \\ D(Q,P)/\ell P \sim 2^u}}
\langle \Delta_P f_1\rangle_{P_j} \frac{T_{P_j,Q_i}}{\mu(Q_i)}\bigg\rvert^2\bigg)^{1/2}\bigg\rVert_{p_1}\,.
\end{align*}
The proof of inequality \eqref{l.m1lem} is finished
by invoking Lemma \ref{hardy_finish} below.
\end{proof}

\begin{lemma}\label{hardy_finish}
For $m\in \{0,1,\ldots, r\}$ and $u\in \{0,1\}$, we have
$\rvert \mathbf{A}_{m,u,i,j}\lvert \le C(r,\upsilon)$.
\end{lemma}

\begin{proof}
For each $k\in \mathbf{Z}$ and $S\in\mathcal{D}_{1,k+u+\theta(u+m)}$,  define
a kernel
\begin{equation}\label{kernel_S_II}
K_S(x,y) =  \sum_{\substack{ P\in\mathcal{G}_{1,k} \\ P\subset S} } 
 \sum_{ \substack{Q\in\mathcal{G}_{2,k-m} \\ D(Q,P)/\ell P \sim 2^u}}
 \mathbf{1}_{Q_i}(x) \cdot \widetilde{T}_{P_j,Q_i}\cdot \mathbf{1}_{P_j}(y)\,,\qquad x,y\in\mathbf{R}^n\,,
\end{equation}
where $\widetilde{T}_{P_j,Q_i}=\mathbf{1}_{Q=Q(P)}\widetilde{T}_{P_j,Q_i}$ is defined by
\[
\frac{T_{P_j,Q_i}}{\mu(P_j)\mu(Q_i)} =  \frac{\widetilde{T}_{P_j,Q_i}}{\mu(S)}\,.
\]
Consider cubes $P$ and $Q$ as in the definition of $K_S$, and
let $y\in P_{j,\textup{sep}}$ and $x\in Q_i$.
By the upper doubling properties of $\mu$, and the facts that
 $\lvert x-y\rvert \ge \upsilon 2^{-r}\ell P_j$ and
$S\subset B(y, 2^{1+u+\theta(u+m)}\ell P)$, we find
that $\lambda(y, \lvert x-y\rvert)^{-1}\le C(r,\upsilon)\mu(S)^{-1}$.
Hence, by
definition,
\begin{align*}
\lvert T_{P_j,Q_i}\rvert  \le \int_{Q_i}\int_{P_{j,\textup{sep}}}
\frac{1}{\lambda(y,\lvert x-y\rvert)}\,d\mu(y)\,d\mu(x)
\le C(r,\upsilon) \mu(Q_i) \mu(P_j)\mu(S)^{-1}\,.
\end{align*}
As a consequence $\lvert \widetilde{T}_{P_j,Q_i}\rvert
\le C(r,\upsilon)$ and,
by recalling Lemma \ref{l.s_exists_s},
\begin{equation}\label{kernel_est}
\begin{split}
\lvert K_S(x,y)\rvert &\le C(r,\upsilon)\sum_{
\substack{P\in\mathcal{G}_{1,k} \\ P\subset S}} 
 \sum_{ \substack{P\in\mathcal{G}_{2,k-m} \\ D(Q,P)/\ell P \sim 2^u}}
 \mathbf{1}_{Q_i}(x) \cdot \mathbf{1}_{P_j}(y)
\le C(r,\upsilon)\cdot \mathbf{1}_S(x) \cdot \mathbf{1}_S(y)\,.
\end{split}
\end{equation}
After these preparations, we finish the proof by proceeding as in
Lemma \ref{l.useful}.
\end{proof}

\subsection{Intersecting term}

The following inequality deals with intersecting part, i.e., terms $\beta_3(P)$;
\begin{equation}\label{alpha_4}
\bigg|
\sum_{\substack{P\in\mathcal{G}_1}} \langle \Delta_P f_1\rangle_{P_j} \beta_3(P)\langle \Delta_Q f_2 \rangle_{Q_i}
\bigg|
\le C(r,\upsilon,\epsilon)(1+\mathbf{T}_{\textup{loc}})\,.
\end{equation}
The proof of this inequality relies on the kernel size
estimate and local testing conditions.

\begin{proof}[Proof of inequality \eqref{alpha_4}]
We tacitly restrict all the summations here to cubes $P\in\mathcal{G}_1$
for which $\mu(Q_i\cap P_j)\not=0$. Indeed, otherwise $\beta_3(P)=0$.
By writing $\mu(Q_i\cap P_j)=\int \mathbf{1}_{Q_i} \mathbf{1}_{P_j}\,d\mu$ and using
Cauchy-Schwarz and H\"older's inequality,
\begin{equation}\label{estnormad}
\begin{split}
&\bigg|\sum_{P\in\mathcal{G}_1} \langle \Delta_P f_1\rangle_{P_j} \beta_3(P)\langle \Delta_Q f_2\rangle_{Q_i}\bigg|\\
&\le \bigg\| \bigg(\sum_{P\in\mathcal{G}_1}  \lvert \langle \Delta_P f_1\rangle_{P_j}\mathbf{1}_{P_j} \rvert^2 \bigg)^{1/2}\bigg\|_{p_1} \cdot \bigg\| \bigg(
\sum_{P\in\mathcal{G}_1} \bigg\lvert
\frac{ \beta_3(P) }{\mu(Q_i\cap P_j)}\langle \Delta_Q f_2\rangle_{Q_i}\mathbf{1}_{Q_i}\bigg\rvert^2 \bigg)^{1/2}\bigg\|_{p_2}.
\end{split}
\end{equation}
By inequality \eqref{kahane}, the first factor is bounded by
$\lesssim 1$.
Let us then focus on the second factor;
by writing the summation in terms of $Q$ and 
using Lemma \eqref{talkaa}, we 
obtain an upper bound $C(r,\upsilon,\epsilon) (1+\mathbf{T}_{\textup{loc}})$
for the second term.
\end{proof}

\begin{lemma}\label{talkaa}
Let $P\in\mathcal{G}_1$. Then
$|\beta_3(P)|\le C(r,\upsilon,\epsilon) (1+\mathbf{T}_{\textup{loc}})\mu(Q_i\cap P_j)$.
\end{lemma}

\begin{proof}
We can assume that $Q=Q(P)\not=\emptyset$, hence $P\sim Q$.
Consider the expansion,
\begin{align*}
\beta_3(P)=\langle T\mathbf{1}_{\widetilde \Delta_{P_j}}, \mathbf{1}_{\widetilde \Delta_{Q_i}}\rangle= \sum_{\substack{G,H\in \mathcal{L}\\G\not=H}} \langle 
T(\mathbf{1}_G \mathbf{1}_{\widetilde \Delta_{P_j}}),
\mathbf{1}_{H}\mathbf{1}_{\widetilde \Delta_{Q_i}}\rangle
+\sum_{G\in \mathcal{L}}
\langle  T(\mathbf{1}_G \mathbf{1}_{\widetilde \Delta_{P_j}}), \mathbf{1}_{G}\mathbf{1}_{\widetilde \Delta_{Q_i}}\rangle.
\end{align*}
In both of the series above, 
the finite number of summands depends on
$n$ and $\upsilon$. Hence, 
it suffices to obtain estimates
for individual summands for fixed $G,H\in\mathcal{L}$.
%
First, if $G\not=H$, then 
\[\ell Q_i\le C(r,\upsilon,\epsilon)\,\mathrm{dist}(G\cap \widetilde \Delta_{P_j},H\cap \widetilde \Delta_{Q_i}).\]
In particular,   $\lambda(x,\lvert x-y\rvert)^{-1} \le C(r,\upsilon,\epsilon)\mu(Q_i)^{-1}$ if $x\in H\cap \widetilde \Delta_{Q_i}$
and $y\in G\cap \widetilde \Delta_{P_j}$.
Hence, 
\begin{equation}\label{kestim}
\begin{split}
|\langle T(\mathbf{1}_G \mathbf{1}_{\widetilde \Delta_{P_j}}), \mathbf{1}_{H}\mathbf{1}_{\widetilde \Delta_{Q_i}}\rangle|
&\le \int_{H\cap \widetilde \Delta_{Q_i}}\int_{G\cap \widetilde \Delta_{P_j}} 
\frac{1}{\lambda(x,\lvert x-y\rvert)}\,d\mu(y)\,d\mu(x)\\
&\le C(r,\upsilon,\epsilon)\frac{\mu(Q_i\cap P_j)\mu(Q_i\cap P_j)}{\mu(Q_i)}\le C(r,\upsilon,\epsilon)\mu(Q_i\cap P_j).
\end{split}
\end{equation}
In the last step, we also used the fact that $\widetilde\Delta_{P_j}\cup \widetilde \Delta_{Q_i}\subset Q_i\cap P_j$.

Then we consider the case of $G=H$.  
By construction,
\begin{equation}\label{dsh}
\langle T(\mathbf{1}_G \mathbf{1}_{\widetilde \Delta_{P_j}}), \mathbf{1}_{G}\mathbf{1}_{\widetilde \Delta_{Q_i}}
\rangle
=\begin{cases}
\langle T\mathbf{1}_{\widetilde G}, \mathbf{1}_{\widetilde G}\rangle,\quad &\text{if }
G=G\cap \Delta_{P_j}^\mathcal{L}=G\cap \Delta_{Q_i}^\mathcal{L}\,;\\
0&\text{otherwise}.
\end{cases}
\end{equation}
In any case, by local testing conditions
$\lvert \langle T(\mathbf{1}_G \mathbf{1}_{\widetilde \Delta_{P_j}}),\mathbf{1}_{G}\mathbf{1}_{\widetilde \Delta_{Q_i}}\rangle \rvert
\le \mathbf{T}_{\textup{loc}} \mu(\widetilde G)\le  \mathbf{T}_{\textup{loc}}\mu(Q_i \cap P_j)$.
\end{proof}

\section{The Nearby-Boundary Term}\label{s.nearby3}

Here we treat the $\epsilon$ and $\upsilon$ boundary terms
by probabilistic arguments.

\subsection{The $\epsilon$-boundary terms}
Following inequality controls summation for $\epsilon$-boundary terms.
Let $t>p_1\vee p_2$ be a positive real number. Then
\begin{equation}\label{aep}
\mathbf{E}_{\omega_{3}} \bigg|
\sum_{\substack{P\in\mathcal{G}_1}} \langle \Delta_P f_1 \rangle_{P_j} \big(\beta_1(P)+\beta_2(P)\big) \langle \Delta_Q f_2\rangle_{Q_i}
\bigg|
\le C(r,\upsilon)\epsilon^{1/t} \mathbf{T}\,.
\end{equation}
The
expectations over the dyadic system $\mathcal{D}_{3}$ are crucial here,
and here only.

We let
$\epsilon=(\varepsilon_k)_{k\in\mathbf{Z}}$ be a sequence
of Rademacher variables, supported on a probability space $(\Omega,\mathbf{P})$. 
We can also 
associate Rademacher variables to $\mathcal{D}_j$-dyadic cubes with $j\in \{1,2\}$:---
fix an injection $R\mapsto j(R):\mathcal{D}_j\to \mathbf{Z}$, and use
notation $\varepsilon_R = \varepsilon_{j(R)}$. 

We rely on  the following improvement of the contraction principle,
\cite[ Lemma 3.1]{MR2491037}.

\begin{proposition}\label{improved}
Suppose that $\{\rho_k\,:\,k\in \mathbf{Z}\}\subset L^t(\widetilde\Omega)$ for some
$\sigma$-finite measure space $(\widetilde\Omega,\widetilde{\mathbf{P}})$ and $t\in (2,\infty)$.
Then, for all complex-valued sequences $(\xi_k)_{k\in\mathbf{Z}}$,
\[
\bigg\| \sum_{k=-\infty}^\infty \varepsilon_k \rho_k\xi_k\bigg\|_{L^t(\widetilde\Omega;L^2(\Omega))}
\lesssim \sup_{k\in\mathbf{Z}} \|\rho_k\|_{L^t(\widetilde\Omega)}\cdot \bigg\|\sum_{k=-\infty}^\infty \varepsilon_k\xi_k\bigg\|_{L^2(\Omega)}\,.
\]
\end{proposition}

\begin{proof}[Proof of inequality \eqref{aep}]
Let us focus on the sum involving the terms $\beta_1(P)$; 
the estimate for the sum involving terms $\beta_2(P)$ is similar.
We randomize and use H\"older's inequality,
\begin{equation}\label{ranes}
\begin{split}
& \bigg|
\sum_{\substack{P\in\mathcal{G}_1}} \langle \Delta_P f_1\rangle_{P_j} \langle T\mathbf{1}_{\Delta'_{P_j}}
,\mathbf{1}_{\Delta_{Q_i}^\mathcal{L}}\rangle\langle \Delta_Q f_2\rangle_{Q_i}\bigg|\\
&=\bigg|\int_\Omega \Big\langle T\Big(\sum_{S\in\mathcal{G}_1} \varepsilon_S
\mathbf{1}_{\Delta'_{S_j}} \langle\Delta_S f_1\rangle_{S_j}\Big),\sum_{P\in\mathcal{G}_1}\varepsilon_P \mathbf{1}_{\Delta_{Q_i}^\mathcal{L}}\langle \Delta_Q f_2\rangle_{Q_i} 
 \Big\rangle d\mathbf{P}(\epsilon)\bigg|\\
&\le \bigg\lVert T\Big(\sum_{S\in\mathcal{G}_1} \varepsilon_S\mathbf{1}_{\Delta'_{S_j}}\langle\Delta_S f_1\rangle_{S_j}\Big)\bigg\|_{L^{p_1}(\Omega\times\mathbf{R}^n)} \bigg\|\sum_{P\in\mathcal{G}_1}\varepsilon_P\mathbf{1}_{\Delta_{Q_i}^\mathcal{L}}\langle \Delta_Q f_2\rangle_{Q_i}\bigg\rVert_{L^{p_2}(\Omega\times\mathbf{R}^n)}.
\end{split}
\end{equation}
Index the very last  summation in terms of $\mathcal{D}_2$. 
This can be done by using our standing assumptions of $P\mapsto Q(P)=Q$. Then, by
the contraction principle and inequality $|\mathbf{1}_{\Delta_{Q_i}^\mathcal{L}}| \le \mathbf{1}_{Q_i}$, we see that
the second factor in the last line of \eqref{ranes} is bounded (up to a constant multiple) by $\lVert f_2\rVert_{p_2}\lesssim 1$.

In order to estimate the first factor in the last line of \eqref{ranes} we first
extract operator norm $\mathbf{T}$. Then we fix
$S\in\mathcal{G}_{1,k}$ with $k\in\mathrm{Z}$.
By
\eqref{rems} and \eqref{sl},
\[
\Delta_{S_j}'\subset L_\epsilon(S,\upsilon)=\bigcup_{G\in \mathcal{L}(S,\upsilon)} \delta_G^\epsilon\subset 
\bigcup_{m=j(\upsilon)+k-r-1}^{j(\upsilon)+k-1} \bigcup_{G\in\mathcal{D}_{3,m}}  \delta^\varepsilon_G=:
\delta^\epsilon(k).
\] 
Hence,  we have $\mathbf{1}_{\Delta'_{S_j}}\le \mathbf{1}_{\delta^\epsilon(k)}\mathbf{1}_{S_j}$.
By the contraction principle and assumption $t\ge p_1$,
\begin{equation}\label{e.crc}
\begin{split}
&\mathbf{E}_{\omega_{3}}\bigg\| \sum_{S\in\mathcal{G}_1} \varepsilon_S\mathbf{1}_{\Delta'_{S_j}}\langle\Delta_S f_1\rangle_{S_j}\bigg\|_{L^{p_1}(\Omega\times\mathbf{R}^n)}
\lesssim
\mathbf{E}_{\omega_{3}}\bigg\|\sum_{k\in\mathrm{Z}} \varepsilon_k\mathbf{1}_{\delta^\epsilon(k)}\sum_{S\in\mathcal{D}_{1,k}} \mathbf{1}_{S_j}\langle \Delta_S f_1\rangle_{S_j}\bigg\|_{L^{p_1}(\Omega\times\mathbf{R}^n)}\\
&\quad\le \bigg(\int_{\mathbf{R}^n} \bigg[\mathbf{E}_{\omega_{3}}\bigg\|\sum_{k\in\mathrm{Z}} \varepsilon_k\mathbf{1}_{\delta^\epsilon(k)}(x)\sum_{S\in\mathcal{D}_{1,k}} \mathbf{1}_{S_j}(x)\langle \Delta_S f_1\rangle_{S_j}\bigg\|^t_{L^{p_1}(\mathbf{\Omega})}\bigg]^{p_1/t}
d\mu(x)\bigg)^{1/p_1}.
\end{split}
\end{equation}
For a fixed $x\in\mathbf{R}^n$, the last integrand evaluated at $x$ is of
the form as in Proposition \ref{improved} with
$\xi_k = \sum_{S\in\mathcal{D}_{1,k}} \mathbf{1}_{S_j}(x) \langle \Delta_S f_1\rangle_{S_j}$.
Moreover, the random variables $\rho_k:=\mathbf{1}_{\delta^\epsilon(k)}(x)$ as functions of
$\omega_3 \in \Omega_3=(\{0,1\}^n)^{\mathbf{Z}}$ belong to $L^t(\Omega_3)$,
and they satisfy
\[
\sup_{k\in\mathrm{Z}}\|\mathbf{1}_{\delta^\epsilon(k)}(x)\|_{L^t(\Omega_3)} = \sup_{k\in\mathrm{Z}}\mathbf{P}_{\omega_3} (\mathbf{1}_{\delta^\epsilon(k)}(x)=1)^{1/t}\le C(r,\upsilon)\epsilon^{1/t}.
\]
Hence, by Proposition \ref{improved} and Khintchine's inequality,
\begin{align*}
LHS\eqref{e.crc}\le 
C(r,\upsilon)
\epsilon^{1/t}\bigg\| \sum_{S\in\mathcal{D}_1} \varepsilon_S \mathbf{1}_{S_j}\langle \Delta_S f_1\rangle_{S_j} \bigg\|_{L^{p_1}(\Omega\times\mathbf{R}^n)}\le
C(r,\upsilon)\epsilon^{1/t}\,.
\end{align*}
The proof is complete.
\end{proof}

\subsection{The $\upsilon$-boundary terms}
The following inequality controls summation of the $\upsilon$-boundary terms.
Let $t>p_1\vee p_2$. Then
\begin{equation}\label{etalem}
\mathbf{E}_{\omega_1}\mathbf{E}_{\omega_2}\bigg|
\sum_{\substack{P\in\mathcal{G}_1}} \langle \Delta_P f_1\rangle_{P_j} (M_2(P)+M_4(P)+\alpha_2(P)+\alpha_3(P)) \langle \Delta_Q f_2\rangle_{Q_i}
\bigg|
\le C(r)\upsilon^{1/t} \mathbf{T}\,.
\end{equation}
Before the proof, let us remark that although
both $\alpha_2(P)$ and $\alpha_3(P)$ depend 
on the random dyadic system $\mathcal{D}_{3}$, the inequality
is uniform over all such systems.

\begin{proof}[Proof of inequality \eqref{etalem}]
First we observe that functions $f_j$  depend on
{\em both} dyadic systems, as they are (essentially) projections
to good cubes. This dependency is not allowed in the 
argument below. Fortunately, this issue can be easily addressed---if $Q(P)\not=\emptyset$ in the series above, we have both
$P\in\mathcal{G}_1$ and $Q=Q(P)\in\mathcal{G}_2$. Then, in particular
$\Delta_P f_1 = \Delta_P \widetilde{f_1}$
and $\Delta_Q f_2 = \Delta_Q \widetilde f_2$. Functions
$\widetilde f_j$ do not depend on the dyadic systems, and we use them to replace
$f_j$'s.

By \eqref{tenf} and \eqref{ekahaj}, $M_2(P)+\alpha_2(P)$ and $M_4(P)+\alpha_3(P)$
are given by
\begin{align*}
\langle T\mathbf{1}_{P_{j,\partial}}, \mathbf{1}_{Q_i}\rangle
-\langle T\mathbf{1}_{\Delta_{P_j}^\partial}, \mathbf{1}_{\Delta_{Q_i}^\mathcal{L}}\rangle;
\quad \langle T\mathbf{1}_{\Delta_{P_j}}, \mathbf{1}_{Q_{i,\partial}}\rangle-\langle T\mathbf{1}_{\Delta_{P_j}}, \mathbf{1}_{\Delta_{Q_i}^\partial}\rangle\,,
\end{align*}
respectively.
Observe that
\begin{equation}\label{ekat}
\begin{split}
&(\mathbf{1}_{P_{j,\partial}}+\mathbf{1}_{\Delta_{P_j}^\partial}) \lesssim \mathbf{1}_{P_{j,\partial}},\quad (\mathbf{1}_{Q_i}+\mathbf{1}_{\Delta_{Q_i}^\mathcal{L}})\lesssim \mathbf{1}_{Q_i}\,,\quad 
\mathbf{1}_{\Delta_{P_j}}\lesssim \mathbf{1}_{P_j},\,\,\,\quad (\mathbf{1}_{Q_{i,\partial}}+\mathbf{1}_{\Delta_{Q_i}^\partial})\lesssim \mathbf{1}_{Q_{i,\partial}}.
\end{split}
\end{equation}
pointwise $\mu$-almost everywhere.
By triangle inequality, it suffices
to estimate the following sums: one involving terms $m(P)\in\{M_2(P),\alpha_2(P)\}$,
and
the other involving terms in $\{M_4(P),\alpha_3(P)\}$. We
focus on the first sum; the second one is estimated
in an analogous manner, using $\mathbf{E}_{\omega_1}$.

By randomizing, using H\"older's inequality, extracting the operator norm of $T$,
and applying the contraction principle with inequalities \eqref{ekat},
\begin{equation}\label{eeka}
\begin{split}
&\mathbf{E}_{\omega_2}\bigg|
\sum_{\substack{P\in\mathcal{G}_1}} \langle \Delta_{P}
\widetilde{f}_1
\rangle_{P_j} m(P) \langle \Delta_Q \widetilde{f}_2\rangle_{Q_i}
\bigg| \\&\lesssim  \mathbf{T}\cdot \mathbf{E}_{\omega_2}
\bigg\{\bigg\| \sum_{S\in\mathcal{G}_1} \varepsilon_S\mathbf{1}_{S_{j,\partial}} 
\langle \Delta_P\widetilde{f}_1\rangle_{S_j}\bigg\|_{L^{p_1}(\Omega\times\mathbf{R}^n)} \bigg\|\sum_{Q\in\mathcal{D}_2}\varepsilon_Q\mathbf{1}_{Q_i}
\langle \Delta_Q\widetilde{f}_2\rangle_{Q_i}\bigg\|_{L^{p_2}(\Omega\times\mathbf{R}^n)}\bigg\}.
\end{split}
\end{equation}
By the contraction
principle, we find that the last factor is $\omega_2$-uniformly bounded by $\lesssim \lVert \widetilde{f}_2\rVert_{p_2}= 1$.
In order to treat the remaining factor, we write
\[
\delta^\upsilon(k)=\bigcup_{m=k-r-1}^{k-1} \bigcup _{Q\in\mathcal{D}_{2,m}} \delta^\upsilon_Q.
\]
By \eqref{rems} and \eqref{e.basic_one},
$\mathbf{1}_{S_{j,\partial}}\le \mathbf{1}_{S_j}\mathbf{1}_{\delta_{Q_i}^\upsilon}\le \mathbf{1}_{S_j}\mathbf{1}_{\delta^\upsilon(k)}$ if $Q=Q(S)$ with $S\in\mathcal{G}_{1,k}$.
Fix $x\in\mathbf{R}^n$. The random variables $\rho_k:=\mathbf{1}_{\delta^\upsilon(k)}(x)$ as functions of
$\omega_2\in (\{0,1\}^n)^\mathrm{Z}$ belong to $L^t((\{0,1\}^n)^\mathrm{Z})$,
\[
\sup_{k\in\mathrm{Z}} \|\mathbf{1}_{\delta^\upsilon(k)}(x)\|_{L^t((\{0,1\}^n)^\mathrm{Z})} = \sup_{k\in\mathrm{Z}}\mathbf{P}_{\omega_2} (\mathbf{1}_{\delta^\upsilon(k)}(x)=1)^{1/t}\le C(r)\upsilon^{1/t}.
\]
Hence, proceeding as in connection with \eqref{e.crc}, we find that
\begin{align*}
\mathbf{E}_{\omega_2}\bigg\| \sum_{S\in\mathcal{G}_1} \varepsilon_S\mathbf{1}_{S_{j,\partial}}\langle \Delta_P\widetilde{f}_1\rangle_{S_j}\bigg\|_{L^{p_1}(\Omega\times\mathbf{R}^n)}\le 
C(r)\upsilon^{1/t}\bigg\|  \sum_{P\in\mathcal{D}_1} \varepsilon_P\mathbf{1}_{P_j}
\langle  \Delta_P\widetilde{f}_1\rangle_{P_j} \bigg\|_{L^{p_1}(\Omega\times\mathbf{R}^n)}.
\end{align*}
The last term is bounded by a constant multiple of
$C(r)\upsilon^{1/t}$.
\end{proof}

\end{document}